\documentclass[11pt]{amsart}

\usepackage{enumerate, amsmath, amsthm, amsfonts, amssymb, xy,  combelow, mathrsfs, graphicx, paralist, fancyvrb, ytableau, braket, xcolor, verbatim}
\usepackage[bookmarks, colorlinks=true, linkcolor=blue, citecolor=blue, urlcolor=blue]{hyperref}

\input xy
\xyoption{all}

\DeclareMathOperator{\quot}{\mathsf{Quot}}

\newtheorem{theorem}{Theorem}
\newtheorem{lemma}{Lemma}
\newtheorem{corollary}{Corollary}
\newtheorem{proposition}{Proposition}

\theoremstyle{definition}
 
\theoremstyle{definition}
\newtheorem{remark}{Remark}
\theoremstyle{definition}
\newtheorem{definition}{Definition} 

\newtheoremstyle{TheoremNum}
        {7pt}{7pt}              %%% space between body and thm
        {\itshape}                      %%% Thm body font
        {}                              %%% Indent amount (empty = no indent)
        {\bfseries}                     %%% Thm head font
        {.}                             %%% Punctuation after thm head
        { }                             %%% Space after thm head
        {\thmname{#1}\thmnote{ \bfseries #3}}%%% Thm head spec
    \theoremstyle{TheoremNum}

\newcommand{\BA}{\mathbb{A}}
\newcommand{\BC}{\mathbb{C}}
\newcommand{\BN}{\mathbb{N}}
\newcommand{\BP}{\mathbb{P}}
\newcommand{\BQ}{\mathbb{Q}}

\newcommand{\CE}{\mathcal{E}}
\newcommand{\CG}{\mathcal{G}}
\newcommand{\CK}{\mathcal{K}}
\newcommand{\CL}{\mathcal{L}}
\newcommand{\CO}{\mathcal{O}}

\newcommand{\tE}{\widetilde{E}}
\newcommand{\tZ}{\widetilde{Z}}

\newcommand{\tCL}{\widetilde{\CL}}
\newcommand{\tlambda}{\widetilde{\lambda}}
  
\newcommand{\fY}{\mathsf{Y}}
\newcommand{\fH}{\mathsf{H}}
\newcommand{\fA}{\mathsf{CH}}

\begin{document}

\parindent=30pt

\baselineskip=17pt
\title[The cohomology of the Quot scheme on a curve as a Yangian rep]{The cohomology of the Quot scheme on a smooth curve as a Yangian representation}

\author{Alina Marian and Andrei Negu\cb{t}}

\begin{abstract}

We describe the action of the shifted Yangian of $\mathfrak{sl}_2$ on the cohomology groups of the Quot schemes of 0-dimensional quotients of a rank $r$ vector bundle on a smooth projective curve. We introduce a commuting family of $r$ operators in the positive half of the Yangian, whose action yields a natural basis of the cohomology of the aforementioned Quot schemes. These commuting operators further lead to formulas for the operators of multiplication by the Segre classes of the universal bundle. 

\end{abstract}

\address{The Abdus Salam International Centre for Theoretical Physics, Strada Costiera 11, 34151 Trieste, Italy \newline \text{ } \qquad Department of Mathematics, Northeastern University, 360 Huntington Avenue, Boston, MA 02115}
\email{amarian@ictp.it}

\vskip.2in

\address{École Polytechnique Fédérale de Lausanne (EPFL), Lausanne, Switzerland \newline \text{ } \qquad Simion Stoilow Institute of Mathematics (IMAR), Bucharest, Romania}

\email{andrei.negut@gmail.com}

\date{July 2026}

\vskip.2in

\maketitle

\vskip.4in

\section{Introduction}

\vskip.3in

\subsection{Quot schemes over curves}

Fix $V \to C$ a rank $r$ locally free sheaf over a smooth complex projective curve $C$. We consider the Grothendieck Quot scheme $\quot_{d}(V)$ parameterizing rank $0$ degree $d$ quotients of $V$: 
$$
0\to E\to V \to F\to 0,\quad \text{rank }F=0, \quad \text{deg }F=d.
$$
We will denote by $\pi, \rho$ the projections from the product $\quot_d(V) \times C$ to the two factors. In any setting, the Quot scheme carries a universal short exact sequence
$$
0 \to \mathcal E \to \rho^* (V) \to \mathcal F \to 0 \, \, \, \text{on}  \, \, \, \quot_d (V) \times C,
$$ 
and its deformation-obstruction complex is $\text{Ext}^\bullet_\pi (\mathcal E, \, \mathcal F)$. In the situation considered here, when the quotients $F$ are supported at finitely many points of $C$, it is therefore clear that $\quot_d(V)$ is a smooth projective variety of dimension $rd$. 

The present paper is concerned with a representation-theoretic interpretation of the {\it cohomology} (with $\mathbb{Q}$ coefficients) of the Quot scheme $\quot_d(V)$. On the one hand, our study explains and recovers the Betti number series  summed over all quotient degrees, paralleling in the curve setting the classic result for the Hilbert scheme of points over a smooth surface (cf. \cite{gottsche}, \cite{nakajima}, \cite{grojnowski}). On the other hand, it connects to the well-studied case of the equivariant cohomology of Quot schemes over $\mathbb A^1$ (the so-called Laumon spaces), realized by \cite{bffr, fffr} as the universal Verma module for the {\it shifted truncated Yangian} of $\mathfrak{sl}_2$:
$$
Y_{\hbar}^r(\mathfrak{sl}_2) = \BQ[\hbar] \Big\langle e_k, f_k, m_i \Big\rangle_{k \geq 0, i \in \{1,\dots,r\}} \Big/ \Big(\text{relations \eqref{eqn:rel mm}-\eqref{eqn:rel ef}} \Big)
$$
We extend the Yangian action to Quot schemes over an arbitrary smooth projective curve.

\subsection{A commuting family of operators}
\label{sub:commuting intro}

In what follows, we write $\fH_X$ for the cohomology ring (with rational coefficients) of a smooth projective variety $X$. A basic geometric object in our analysis is the nested Quot scheme 
$$\quot_{d, d+1} (V) \subset \quot_d(V) \times \quot_{d+1}(V),$$
\begin{equation}
  \quot_{d, d+1}(V) = \{ (E \overset{\iota} {\hookrightarrow} V, \, \, \, \, E' \overset{\iota'}\hookrightarrow V)\, \,  \text{with } E'\overset{\kappa} \hookrightarrow E \text{ and }  \iota' = \iota \circ \kappa\}.
\end{equation}
It is endowed with maps
 \begin{equation}
\label{eqn:diagram zk}
\xymatrix{& \quot_{d,d+1}(V) \ar[ld]_{p_-} \ar[d]^{p_C} \ar[rd]^{p_+} & \\ \quot_{d}(V) & C & \quot_{d+1}(V)}
\end{equation}
which remember $E \hookrightarrow V,$ the support point of $E/{E'},$ and $E' \hookrightarrow V$ respectively. The nested Quot scheme can be viewed as parameterizing points $E \subset V$ in $\quot_d (V)$, along with non-zero morphisms $E \to \mathbb C_x$ for $x \in C$. 
We thus have the isomorphism 
 \begin{equation}
 \label{eqn:quot proj 1}
 \xymatrix{\quot_{d, d+1} (V) \ar[r]^{\cong} \ar[rd]_-{p_{-}  \times \, p_C} & \mathbb P_{\quot_d (V) \times C}  \, (\mathcal E) \ar[d] \\  &\quot_d \times \,C}
  \end{equation}
where the right-hand side is the projective bundle of one-dimensional quotients of the universal subsheaf $\mathcal E \to \quot_d (V) \times C.$ We note that $\mathcal E$ is a locally free sheaf of rank $r$. We consider the tautological sequence
\begin{equation}
\label{secondexact}
0 \to \mathcal G \to (p_{-}  \times p_C)^* (\mathcal E) \to \mathcal L \to 0 \, \, \, \text{on} \, \, \, \mathbb P_{\quot_d (V) \times C} (\mathcal E),
\end{equation}
where $\mathcal L$ is the hyperplane line bundle of the projectivization, and the kernel $\mathcal G$ is a locally free sheaf of rank $r-1$. We write
$$
c_1 (\mathcal L) =\lambda,
$$ 
and consider the creation/annihilation operators
\begin{equation}
\label{eqn:formula e intro}
e_k = (p_+ \times p_C)_* ( \lambda^k \cdot p_-^* ) : \fH_{\quot_d(V)} \rightarrow \fH_{\quot_{d+1}(V) \times C},
\end{equation}
\begin{equation}
\label{eqn:formula f intro}
f_k = (p_{-} \times p_C)_* ( \lambda^k \cdot p_+^* ) : \fH_{\quot_{d+1}(V)} \rightarrow \fH_{\quot_{d}(V) \times C}.
\end{equation}
for all $k,d \geq 0$. Operators closely related to $e_0$ and $f_0$ were studied in the symmetric product context in \cite{polishchuk}. It is also natural to use the Chern classes of $\mathcal G$ in order to define the operators
\begin{equation}
\label{eqn:formula a intro}
a_k = (p_+ \times p_C)_* \Big( c_k(\mathcal{G}) \cdot p_-^* \Big) : \fH_{\quot_d(V)} \rightarrow \fH_{\quot_{d+1}(V) \times C}
\end{equation}
for $k \in \{0,\dots,r-1\}$ and all $d \geq 0$. Our first result is the following.

\begin{proposition}
\label{prop:commute a intro}

For any $k,k' \in \{0,\dots,r-1\}$ and $d \geq 0$, we have
\begin{equation}
\label{eqn:commute a intro} 
a_k  a_{k'} = a_{k'} a_k
\end{equation}
as operators $\fH_{\quot_d(V)} \rightarrow \fH_{\quot_{d+2}(V) \times C \times C}$, with the operator denoted by $a_k$ (respectively $a_{k'}$) acting on both sides in the first (respectively second) factor of $C \times C$.

\end{proposition}

\subsection{A natural basis for the cohomology}

As a consequence of Proposition \ref{prop:commute a intro}, we note that products such as
$$
a_{k_1} \dots a_{k_n} : \fH_{\quot_d(V)} \rightarrow \fH_{\quot_{d+n}(V) \times C^n}
$$
are independent of the order of $k_1,\dots,k_n \in \mathbb{N}$ (it is implied that if one permutes $a_{k_i}$ and $a_{k_j}$ for some $i\neq j$, one should also permute the $i$-th and $j$-th factors of $C$ in the product $C^n$). Therefore, we may assume that $k_1 \geq \dots \geq k_n$, and define 
$$
a_{k_1} \dots a_{k_n}(\gamma) : \fH_{\quot_d(V)} \rightarrow \fH_{\quot_{d+n}(V)}
$$
for any $\gamma \in \fH_{C^n}$ to be the composition
\begin{equation}
\label{eqn:a gamma intro}
\fH_{\quot_d(V)} \xrightarrow{a_{k_1} \dots a_{k_n}} \fH_{\quot_{d+n}(V) \times C^n} \xrightarrow{\cdot \rho^*(\gamma)} \fH_{\quot_{d+n}(V) \times C^n} \xrightarrow{\pi_*} \fH_{\quot_{d+n}(V)}
\end{equation}
where $\pi : \quot_d(V) \times C^n \rightarrow \quot_d(V)$ and $\rho : \quot_d(V) \times C^n \rightarrow C^n$ denote the two projections. Let $|0\rangle$ denote the fundamental class of $\quot_0(V) = \text{point}$. We show

\begin{theorem}
\label{theorem:fock basis}

For any $d \geq 0$, we have
\begin{equation}
\label{eqn:fock basis}
\fH_{\quot_d(V)} = \bigoplus^{r > k_1 \geq \dots \geq k_d \geq 0}_{\gamma \,\in\, \,  \BQ\text{-basis of } \fH_{C^d,\Sigma}} \mathbb{Q} \cdot a_{k_1} \dots a_{k_d}(\gamma)|0\rangle
\end{equation}
where $\fH_{C^d,\Sigma}$ denotes the space of coinvariants in $\fH_{C^d}$ under the permutation of the $i$-th and $j$-th factors of $C^d$ for any $i$ and $j$ such that $k_i = k_j$.

\end{theorem}

Taking the direct sum of \eqref{eqn:fock basis} over all $d \geq 0$ allows us to recover the well-known formula for the Poincar\'{e} polynomials of Quot schemes of rank zero quotients over curves:
$$
P_{\quot(V)}(t,z) = \sum_{d=0}^{\infty} t^d P_{\quot_d(V)}(z) = \sum_{d=0}^{\infty} t^d \sum_{k=0}^{2rd} z^k \dim_{\mathbb{Q}}(\fH^k_{\quot_d(V)}) = 
$$
\begin{equation}
\label{eqn:poincare}
= \prod_{i=0}^{r-1} \frac{(1 + t z^{2i+1})^{2g}}{(1-t z^{2i}) (1- t z^{2i+2})}.
\end{equation}
where $\fH^k_X$ denotes the degree $k$ part of $\fH_X$. Formula \eqref{eqn:poincare} was established in \cite{stromme}, \cite{bgl}, \cite{chen}, and recovered by a motive calculation in \cite{bfp}, \cite{ricolfi}. Theorem \ref{theorem:fock basis} gives a representation-theoretic interpretation of this formula (mirroring the analysis for Hilbert schemes of surfaces of \cite{grojnowski, nakajima}). 

\subsection{Multiplication operators}

To prove Theorem \ref{theorem:fock basis}, we establish two intersection theoretic results, which are neatly packaged in terms of the operators $\{a_k,f_k\}_{k \in \{0 \dots, r-1\}}$. The first of these results is a formula for the  operators of multiplication by Segre classes of the universal sheaf following those developed in \cite{lehn} for Hilbert schemes of surfaces and in \cite{negut3} for moduli spaces of higher rank stable sheaves on surfaces. 

\begin{proposition}
\label{prop:mult intro}

For any $k > 0, d \geq 0$, the operator $\fH_{\quot_d(V)} \rightarrow \fH_{\quot_d(V) \times C}$ of multiplication with the class
\begin{equation}
\label{eqn:classes intro}
c_k( (V-\mathcal{E}) \otimes \mathcal{K}_C^{-1})
\end{equation}
is given by
\begin{equation}
\label{eqn:equality coefficients intro}
\sum_{i=0}^{r-1} a_i f_{r+k-i-2} \Big|_\Delta   (-1)^{i-k}
\end{equation}
Here the composition $a_i f_{r+k-i-2}$ is an operator $\fH_{\quot_d(V)} \to \fH_{\quot_d(V) \times C \times C}$, and $|_\Delta$ means restricting the codomain to the diagonal $\Delta : C \hookrightarrow C \times C$.
    
\end{proposition}

More generally, the operator of multiplication by any polynomial in the Chern classes of $\mathcal{E}$ will be a suitable sum of products of the operators \eqref{eqn:equality coefficients intro}, and these can be put in canonical form by commuting all the $f$'s to the right of all the $a$'s. To compute such commutators in practice, we will need the following result. We will write $\delta \in \fH^2_{C \times C}$ for the class of the diagonal, and $\chi_{i+j,r-1}$ for the Kronecker delta symbol (equal to 1 if $i+j=r-1$ and 0 otherwise).

\begin{proposition}
\label{prop:commute intro}

For any $i,j \in \{0,\dots,r-1\}$, we have the following relation
\begin{equation}
\label{eqn:comm a and f intro}
[f_j,a_i] = \delta \cdot\begin{cases} \chi_{i+j,r-1} (-1)^i + \sum_{s=0}^{i-1} a_s f_{i+j-s-1}  (-1)^{i-s} &\text{if } i+j \leq r-1 \\ - \sum_{s=i}^{r-1} a_s f_{i+j-s-1} (-1)^{i-s} &\text{if } i+j \geq r \end{cases} 
\end{equation}
The LHS of \eqref{eqn:comm a and f intro} is an operator $\fH_{\quot_d} \rightarrow \fH_{\quot_d \times C \times C}$, with $a_i$ and $f_j$ acting in the first and second factor of $C \times C$, respectively. The analogous convention applies to the RHS of \eqref{eqn:comm a and f intro}, even though the presence of $\delta$ ensures that the overall expression is supported on the diagonal of $C \times C$.

\end{proposition}

In Section \ref{sec:theorem 1} we will use Propositions \ref{prop:mult intro} and \ref{prop:commute intro} (together with the well-known fact that any element of $\fH_{\quot_d(V)}$ is a polynomial in the Chern classes of the universal sheaf) to show that the classes
$$
\Big\{ a_{k_1} \dots a_{k_d}(\gamma)|0\rangle \Big\}_{r > k_1 \geq \dots \geq k_d \geq 0, \, \, \gamma \in \, {\mathbb Q}{\text{-basis of} \, \fH_{C^d,\Sigma}}}
$$
span $\fH_{\quot_d(V)}$ and are linearly independent, thus establishing Theorem \ref{theorem:fock basis}.

\subsection{The Yangian action}
Recall the operators $e_k, \, f_k$ of \eqref{eqn:formula e intro} and \eqref{eqn:formula f intro}, and let us combine them into the following formal series:
$$
e(z) = \sum_{k=0}^{\infty} \frac {e_k}{z^{k+1}}, \qquad f(z) = \sum_{k=0}^{\infty} \frac {f_k}{z^{k+1}}.
$$
Consider the operators of multiplication by universal classes
\begin{equation}
\label{eqn:operator m intro}
m_i = (\text{multiplication by } c_i(\mathcal{E})) \, \circ \pi^* : \fH_{\quot_d} \rightarrow \fH_{\quot_{d} \times C}.
\end{equation}
Let us also combine them into the following degree $r$ monic polynomial
$$
m(z) = \sum_{i=0}^r z^{r-i} (-1)^i m_i,
$$
which represents multiplication by the universal Chern series 
$$
c (\mathcal E, z) =\sum_{i=0}^r z^{r-i} (-1)^i c_i (\mathcal E). 
$$
Note that $c(\mathcal E, z)$ is $z^r$ times a polynomial in $z^{-1}$, which can be formally inverted in the ring of Laurent series in $z^{-1}$. With this in mind, we may consider the power series in $z^{-1}$
\begin{equation}
\label{eqn:abstract h}
h(z) = \sum_{k=0}^{\infty} \frac {h_k}{z^{k+r}} := \frac {c(V, z + K_C)}{c(\mathcal E, z) c(\mathcal E, z+K_C)} 
\end{equation}
where $K_C \in \fH^2_C$ denotes the canonical class. We make the convention that $h_k = 0$ if $k < 0$. Recall the diagonal class $\delta = [ \Delta] \in \fH^2_{C \times C}$, and let
\begin{equation}
{\mathsf H}_{\quot \times C^n} = \bigoplus_{d=0}^{\infty} \fH_{\quot_d \times C^n}, 
\end{equation}
for any $n \in \BN$.

\begin{theorem}
\label{theorem:yangian}
The following identities hold:
\begin{equation}
\label{eqn:rel mm_quot}
    [m(z), \, m(w)] = 0
\end{equation}
\begin{equation}
\label{eqn:rel ee_quot}
 \Big[ e(z)e(w) (z-w+\delta) \Big]_{z^{<0},w^{<0}}= \Big[ e(w) e(z) (z-w-\delta) \Big]_{z^{<0},w^{<0}}
\end{equation}
\begin{equation}
\label{eqn:rel ff_quot}
\Big[ f(z)f(w) (z-w-\delta) \Big]_{z^{<0},w^{<0}} = \Big[ f(w) f(z) (z-w+\delta) \Big]_{z^{<0},w^{<0}}
\end{equation}
\begin{equation}
\label{eqn:rel me_quot}
m(w) e(z) = \left[ e(z)m(w) \cdot \frac {w-z+\delta}{w-z} \right]_{w \gg z | z^{<0}, w^{\geq 0}}
\end{equation}
\begin{equation}
\label{eqn:rel fm_quot}
f(z)m(w) = \left[ m(w) f(z) \cdot \frac {w-z+\delta}{w-z}\right]_{w \gg z| z^{<0}, w^{\geq 0}}
\end{equation}
\begin{equation}
\label{eqn:rel ef_quot}
[e(z), f(w)] = \delta \cdot \frac {h(z)-h(w)}{z-w}
\end{equation}
where subscripts such as $z^{<0},w^{<0}$ means that we only retain those powers of $z,w$ prescribed by the inequalities, while the subscript $w \gg z$ means that we expand all rational functions in $z,w$ in the prescribed order. Formulas \eqref{eqn:rel mm_quot}--\eqref{eqn:rel ef_quot} are equalities of formal power series whose coefficients in $z, w$ are linear operators 
$$
{\mathsf H}_{\quot} \rightarrow {\mathsf H}_{\quot \times C \times C}.
$$
We make the convention that on both sides of each equation \eqref{eqn:rel mm_quot}--\eqref{eqn:rel ef_quot}, the operators in the series with variable $w$ (respectively $z$) are associated with the first (respectively second) factor of $C \times C$.
\end{theorem}

\begin{remark}
From the tautological sequence \eqref{secondexact}, we note that if we assemble the $a$ operators as $$a (z) = \sum_{i=0}^{r-1} z^{r-i-1} (-1)^i a_i,$$
we have 
\begin{equation}
\label{eqn:formula a in terms of e first}
a(z) = e(z) m(z) \Big|_\Delta
\end{equation}
with the composition $e(z)m(z): \fH_{\quot_d} \rightarrow \fH_{\quot_{d+1} \times C \times C}((z^{-1}))$, and $|_\Delta$ denoting restriction to the diagonal of $C \times C$ (see Remark \ref{rem:ae} for details). The commuting operators $a$ are thus part of the larger algebra considered above.
\end{remark}

In terms of series coefficients, formulas \eqref{eqn:rel mm_quot}--\eqref{eqn:rel ef_quot} are equivalent to the following commutator relations for all $k,l \geq 0$:
\begin{equation}
\label{eqn:rel mm_quot coefficient}
[m_k,m_l] = 0 
\end{equation}
\begin{equation}
\label{eqn:rel ee_quot coefficient}
[e_{k+1},e_l] -[ e_k, e_{l+1}] = (-\delta) \left( e_k e_l + e_l e_k \right)  
\end{equation}
\begin{equation}
\label{eqn:rel ff_quot coefficient}
[f_{l},f_{k+1}] -[ f_{l+1}, f_{k}] = (-\delta) \left( f_k f_l  + f_l f_k  \right) 
\end{equation}
\begin{equation}
\label{eqn:rel me_quot coefficient}
[m_l, e_k] = \delta \sum_{s = 0}^{l-1} e_{k+s} m_{l-s-1} (-1)^{s+1} 
\end{equation}
\begin{equation}
\label{eqn:rel fm_quot coefficient}
[f_k , m_l] = \delta \sum_{s = 0}^{l-1}  m_{l-s-1} f_{k+s} (-1)^{s+1} 
\end{equation}
\begin{equation}
\label{eqn:rel ef_quot coefficient}
[e_k, f_l] = (-\delta) \cdot h_{k+l-r+1} 
\end{equation}
As before, we make the convention that the operators with labels $k \pm s$ (respectively $l\pm s$) take values in the first (respectively second) factor of $C \times C$, in every equation. Formulas \eqref{eqn:rel mm_quot coefficient}--\eqref{eqn:rel ef_quot coefficient} match the defining relations in the shifted Yangian on $\mathfrak{sl}_2$ (defined in greater generality in \cite{bk}), which we now recall.

\begin{definition}
\label{def:yangian}
Let $\hbar$ be a formal parameter and fix $r \in \mathbb N$. The shifted Yangian 
$Y_{\hbar}^r(\mathfrak{sl}_2)$ is the $\BQ[\hbar]$ algebra generated by symbols
$$
\Big\{ e_k, f_k, m_i \Big\}_{k \geq 0, i \in \{1,\dots,r\}}
$$
modulo the following relations for all $k,l \geq 0$:
\begin{equation}
\label{eqn:rel mm}
[m_k,m_l] = 0 
\end{equation}
\begin{equation}
\label{eqn:rel ee}
\frac {[e_{k+1},e_l]}{\hbar} -\frac {[ e_k, e_{l+1}]}{\hbar} =  e_k e_l + e_l e_k \end{equation}
\begin{equation}
\label{eqn:rel ff}
\frac {[f_{l},f_{k+1}]}{\hbar} - \frac {[ f_{l+1}, f_{k}]}{\hbar} = f_k f_l + f_l f_k \end{equation}
\begin{equation}
\label{eqn:rel me}
\frac {[m_l, e_k]}{\hbar} = - \sum_{s = 0}^{l-1} e_{k+s} m_{l-s-1} (-1)^{s+1}  
\end{equation}
\begin{equation}
\label{eqn:rel fm}
\frac {[f_k , m_l]}{\hbar} = -\sum_{s = 0}^{l-1}  m_{l-s-1} f_{k+s} (-1)^{s+1} 
\end{equation}
\begin{equation}
\label{eqn:rel ef}
\frac {[e_k, f_l]}{\hbar} = h_{k+l-r+1} 
\end{equation}
where
\begin{equation}
\label{eqn:abstract h yangian}
h(z) = \sum_{k=0}^{\infty} \frac {h_k}{z^{k+r}} := \frac {P(z)}{m(z)m(z+\hbar)}
\end{equation}
for some degree $r$ monic polynomial $P(z)$ (we declare $h_k = 0$ for $k<0$). 

\end{definition}

We note that ``shifted" refers to the series \eqref{eqn:abstract h yangian} starting at $z^{-r}$. Comparing the definition of the Yangian algebra with the identities of Theorem \ref{theorem:yangian}, it is clear we can define a $\mathbb Q$-linear map
\begin{equation}
\label{eqn:yangian rep}
Y_{\hbar}^r(\mathfrak{sl}_2) \longrightarrow \, \text{Hom} \, ({\mathsf H}_{\quot}, \, {\mathsf H}_{\quot \times C} ) 
\end{equation}
such that the Yangian generators are mapped to the geometric operators $e, f, m$, the formal variable $\hbar$ is sent to $K_C$, and the multiplication in the Yangian corresponds in the right hand side of \eqref{eqn:yangian rep} to composition of operators followed by restriction to the diagonal divisor $\Delta \subset C \times C$ (see Definition \ref{def:action} in Section \ref{sec:rep theory} for a more precise statement of the properties of the map \eqref{eqn:yangian rep}, and for the reason it deserves to be called an ``action").

\begin{corollary}
\label{cor:yangian}
For any rank $r$ vector bundle $V$ on a smooth projective curve $C$, the map
$$
Y_{\hbar}^r(\mathfrak{sl}_2) \rightarrow \emph{Hom} \, ({\mathsf H}_{\quot}, \, {\mathsf H}_{\quot \times C} ) 
$$
determined by \eqref{eqn:formula e intro}, \eqref{eqn:formula f intro} and \eqref{eqn:operator m intro} is an action in the sense of Definition \ref{def:action}.

\end{corollary}

\subsection{Other directions} 

While our results pertain to the cohomology of Quot schemes on a smooth projective curve $C$, our approach admits a number of natural generalizations.

\begin{itemize}

\item One can drop the condition that $C$ be projective and/or work with torus equivariant cohomology, and all our results still hold. In fact, Theorem \ref{theorem:yangian} is modeled on the fundamental case of $C$ = equivariant $\mathbb{A}^1$ of \cite{bffr,fffr}.

\item One can replace cohomology by Chow groups. We will provide details about this modification in Section \ref{sec:chow}. Most of our main results (Theorem \ref{theorem:yangian}, Propositions \ref{prop:commute a intro}, \ref{prop:mult intro} and \ref{prop:commute intro}, and Corollary \ref{cor:yangian}) continue to hold. As for Theorem \ref{theorem:fock basis}, our argument shows that the various $a_{k_1} \dots a_{k_d}(\gamma)|0\rangle$ span the Chow groups of $\quot_d(V)$. However the proof of their linear independence fails at the level of Chow groups (this is because Proposition \ref{prop:linearly independent} uses the non-degeneracy of the intersection pairing).

\item At the cost of working with the derived Quot schemes of \cite{ck}, one can consider the situation when the curve $C$ is singular. While we make no claims about the technicalities involved, we expect the approach and results in the present paper to carry over to the case of singular $C$. Note that the operators considered in \cite{rennemo,kivinen} in connection to DT theory and algebraic knots are particular cases of our operators when $V = \mathcal{O}_C$.

\item After the present paper was written, a sophisticated generalization of our results to the setting of hyperquot schemes (parametrizing $n$-step flags of finite length quotients on $C$) was carried out by Kaushik in \cite{kaushik}.  The shifted Yangian of $\mathfrak{sl}_2$  is replaced by the shifted Yangian of $\mathfrak{sl}_{n+1}$, which also aligns with the treatment of parabolic Laumon spaces in \cite{bffr}. Overall the geometry developed in \cite{kaushik} is intricate, as the analysis of the Yangian action in the hyperquot setting is far from a routine generalization of the $n=1$ case of our paper.  

\end{itemize}

\subsection{Context and previous work} Over the past few decades, there has been substantial progress toward understanding the intersection theory of the Quot scheme parametrizing quotients of a locally free sheaf over a smooth curve. It is known that the Quot scheme admits a virtual class for any quotient rank and degree (cf. \cite{mo1}). Several associated virtual invariants have been analyzed: intersections of universal Atiyah-Bott generators in \cite{bertram}, \cite{mo1}, Segre invariants of tautological bundles in \cite{op}, K-theoretic invariants in \cite{os}, \cite{mos}. Connections to the intersection theory of the stack of vector bundles on $C$ were developed in \cite{witten}, \cite{bdw}, \cite{alina}, \cite{mo2}, \cite{mo3}. Relations to the virtual intersection theory of the Quot scheme of rank $0$ quotients on a surface were recently explored in \cite{op}, \cite{lim}, \cite{ajlop}, \cite{stark}. 
 
In a different direction, the study of cohomology and $K$-theory through the action of correspondences given by inclusion of sheaves has had much success in the case of Hilbert schemes of points on surfaces in \cite{nakajima}, \cite{grojnowski} and also for higher rank sheaves on surfaces (cf. \cite{baranovsky}, \cite{negut1, negut2, negut3, negut4}). In the curve setting, the Laumon/hyperquot space of full flags of subsheaves  of $\mathbb C^r$ on $\mathbb P^1$, compactifying the space of maps from $\mathbb P^1$ to the full flag variety of $\mathbb C^r$, was analyzed from this viewpoint in \cite{fk}; natural Hecke correspondences of sheaves were shown to give an $\mathfrak{sl}_r$ action on the cohomology of the full-flag hyperquot schemes, summing over all degree vectors. Related aspects of the $\mathfrak{sl}_r$ action on the {\it{equivariant}} cohomology for a maximal torus action were further studied in \cite{fffr}, where the extension to Yangians was also introduced. We note also the explicit description of the equivariant cohomology ring of the Quot scheme in \cite{bcs}.

\subsection{Acknowledgements} 

A.M. was supported by the NSF through grant DMS-1902310. A.N. gratefully acknowledges NSF grant DMS-$1845034$, as well as support from the MIT Research Support Committee. A.M. is indebted to  Anthony Licata for inspiring conversations surrounding 
algebra actions on the cohomology of Quot schemes several years ago.  A.N. would like to thank Alexander Tsymbaliuk for many interesting discussions about Yangians.

\section{The geometry of (nested) Quot schemes}
\label{sec:geometry}

\subsection{Projective bundles and pushforward formulas} If $\mathsf A$ is a rank $r$ vector bundle on a variety $X$, we will write
$$
c(\mathsf A,z) = \sum_{i=0}^r (-1)^i c_i(\mathsf A) \cdot z^{r-i}
$$
for the Chern polynomial of $\mathsf A$. We extend the notation above multiplicatively to $K$-theory classes, i.e. formal differences $\mathsf A - \mathsf B$ of vector bundles on $X$. 

\begin{definition}

If $\mathsf A$ is a rank $r$ vector bundle on an algebraic variety $X$, we write
\begin{equation}
\label{eqn:def proj 1}
\mathbb{P}_X(\mathsf A) = \text{Proj}_X(\text{Sym}^\bullet \mathsf A)
\end{equation}
for the projective bundle of one-dimensional quotients of $\mathsf A$. It is endowed with a tautological line bundle $\mathcal{O}(1)$ and a tautological morphism
$$
\pi^*(\mathsf A) \twoheadrightarrow \mathcal{O}(1)
$$
where $\pi : \mathbb{P}_X(\mathsf A) \rightarrow X$ is the natural projection.
\end{definition}

We note the basic pushforward formula
\begin{equation}
\label{eqn:segre 1}
\pi_*\left[ \frac 1{z-\lambda}\right] = \frac 1{c(\mathsf A, z)}
\end{equation}
(as power series in $z^{-1}$) where
$$
\lambda = c_1(\mathcal{O}(1)) \in \fH^2_{\mathbb{P}_X(\mathsf A)}.
$$
It is easy to see that \eqref{eqn:segre 1} implies the formula 
\begin{equation}
\label{eqn:segre 1 modified}
\pi_*\left[ \frac {P(\lambda)}{z-\lambda}\right] = \left[ \frac {P(z)}{c(\mathsf A, z)} \right]_{z^{<0}}
\end{equation}
for every polynomial $P$, where the notation $[\dots]_{z^{<0}}$ means that we expand in non-positive powers of $z$ and remove all terms of the form $z^n$ for $n\geq 0$.

\begin{definition}

Consider a morphism of vector bundles on an algebraic variety $X$
$$
\mathsf B \xrightarrow{\phi} \mathsf A
$$
We define the projectivization
\begin{equation}
\label{eqn:def proj 2}
\mathbb{P}_X(\mathsf A- \mathsf B) \stackrel{\iota}\hookrightarrow \mathbb{P}_X(\mathsf A)
\end{equation}
(the morphism $\phi$ is not part of the notation, but it is implied) as the derived zero locus of the composition
$$
\pi^*(\mathsf B) \xrightarrow{\pi^*(\phi)} \pi^*(\mathsf A) \twoheadrightarrow \mathcal{O}(1)
$$
The derived zero locus coincides with the actual zero locus precisely when the latter has the expected dimension $\dim X + \text{rank } \mathsf A - \text{rank }\mathsf B -1$. We refer to \cite{J} for a thorough treatment of projectivizations of complexes in the language of derived algebraic geometry.

\end{definition}

Let $\pi' = \pi \circ \iota : \mathbb{P}_X(\mathsf A - \mathsf B) \rightarrow X$ be the natural projection. By combining \eqref{eqn:segre 1} and \eqref{eqn:segre 1 modified}, we obtain
\begin{equation}
\label{eqn:segre 2}
\pi'_*\left[ \frac 1{z-\lambda}\right] = \left[\frac {c(\mathsf B, z)}{c(\mathsf A, z)} \right]_{z^{<0}}.
\end{equation}

\subsection{Correspondences}
\label{sub:corr}

Although for convenience we often use the language of linear maps, we note that all our calculations are carried out on the level of correspondences. Specifically, given two smooth projective varieties $X$ and $Y$, to a class $\gamma \in \fH_{X \times Y}$, we associate the linear map 
$$
p_{2*}(\gamma \cdot p_1^*): \, \fH_X \rightarrow \fH_Y, 
$$
where $p_1:X \times Y \to X$ and $p_2:X \times Y \to Y$ are the usual projections. Under this identification, composition of linear maps matches the following composition of correspondences
$$
\gamma \in \fH_{X \times Y}, \gamma' \in \fH_{Y \times Z} \leadsto \gamma' \circ \gamma := p_{13*}(p_{12}^*(\gamma) \cdot p_{23}^*(\gamma'))\in \fH_{X \times Z}
$$
with $p_{ij}$ denoting the obvious projection from $X \times Y \times Z$ to two of the three factors. 

\begin{remark}

While we state many of our results (specifically, the ones in Theorem \ref{theorem:yangian} and Propositions \ref{prop:commute a intro}, \ref{prop:mult intro}, \ref{prop:commute intro}) as equalities of operators on cohomology groups, this is only a matter of convenience. They actually hold as equalities of correspondences, which allows them to extend to more general cohomology theories such as Chow groups (see Subsection \ref{sub:chow}).

\end{remark}
 
\subsection{The Quot scheme and universal classes}
\label{sub:quot}

Let $C$ be a smooth projective curve, and fix a locally free sheaf $V \to C$. We will write $r = \text{rank }V$, and henceforth drop $V$ from the notation of all our Quot schemes. For any $d \geq 0$, we consider the Quot scheme
$$
\quot_d = \Big\{0 \to E \to V \to F \to 0, \, \deg F = d, \, \text{rank }F = 0 \Big\}
$$
It is well-known and easy to see that $\quot_d$ is smooth and projective of dimension $rd$. We recall the universal short exact sequence
\begin{equation}
\label{eqn:ses}
0 \rightarrow \mathcal{E} \rightarrow V \rightarrow \mathcal{F} \rightarrow 0.
\end{equation}
on $\quot_d \, \times \, C$, where we abuse notation by writing $V$ for the pull-back of our fixed locally free sheaf $V$ from $C$ to $\quot_d \times C$ via the natural projection. Note that $\CE$ is locally free of rank $r$.

Using the Chern classes $c_k(\mathcal{E}) \in \fH^{2k}_{\quot_d \times C}$ of the universal subbundle $\mathcal E$, for any $k_1,\dots,k_n \in \mathbb{N}$ and any $\gamma \in \fH_{C^n}$, we may consider the pushforward class 
\begin{equation}
\label{eqn:universalclass}
\pi_* \Big(c_{k_1}(\mathcal{E}_1) \dots c_{k_n}(\mathcal{E}_n) \rho^*(\gamma) \Big) \in \fH_{\quot_d}.
\end{equation}
Here $\pi : \quot_d \times C^n \to \quot_d$ and $\rho: \quot_d \times C^n \to C^n$ denote the standard projections, and we let $\mathcal{E}_1,\dots,\CE_n$ be the universal sheaves on $\quot_d \, \times \, C^n$ pulled back from the factors of $C^n$.
We refer to the pushforwards \eqref{eqn:universalclass} as {\it universal classes}.

\begin{remark}
\label{remark:diagonal}

Consider the following coherent sheaves on the product $\quot_d \times \quot_d \times C$:
\begin{align*} 
&\mathcal E \text{ is the universal subsheaf on the first factor} \\
&\mathcal F'\text{ is the universal quotient sheaf on the second factor} 
\end{align*}
Consider the following composition of morphisms of coherent sheaves induced by \eqref{eqn:ses}:
\begin{equation}
\label{eqn:s}
s : \mathcal E \hookrightarrow \rho^*(V) \twoheadrightarrow \mathcal F'
\end{equation}
where we write $\pi: \quot_d \times \quot_d \times C \to \quot_d \times \quot_d$ and $\rho: \quot_d \times \quot_d \times C \to C$ for the standard projections. Because $\mathcal E$ is locally free, we can identify $s$ with a section
$$
\sigma \in \Gamma(\quot_d \times \quot_d \times C, \mathcal E^\vee \otimes \mathcal F')
$$
and it gives rise to a section
$$
\pi_*(\sigma) \in \Gamma(\quot_d \times \quot_d, \pi_*(\mathcal E^\vee \otimes \mathcal F'))
$$
It is clear that $\pi_*(\sigma)$ vanishes precisely on the diagonal of $\quot_d \times \quot_d$, and it is easy to show that $\pi_*(\mathcal E^\vee \otimes \mathcal F')$ is a rank $rd$ locally free sheaf. Together, these facts imply the following formula for the fundamental class of the diagonal:
\begin{equation}
\label{eqn:diagonal}
[\Delta_{\quot_d}] = c_{rd} \left (\pi_* (\mathcal E^\vee \otimes \mathcal F')\right ) \in \fH^{2rd} _{\quot_d \times \quot_d},
\end{equation}
By standard arguments (see for example \cite{Ma}), it follows that the cohomology of $\quot_d$ is spanned by the universal classes \eqref{eqn:universalclass}. 
\end{remark}

\subsection{Nested Quot schemes and projectivizations} 
\label{sub:nested}

Define the nested Quot scheme as the moduli space of sequences
\begin{equation}
\label{eqn:def nested quot scheme}
\quot_{d,d+1} = \Big\{ E' \stackrel{x}\hookrightarrow E \subset V\Big\} \subset \quot_d \times \quot_{d+1},
\end{equation}
with the right-most inclusion having colength $d$. We write throughout
\begin{equation}
\label{eqn:notation x}
E' \stackrel{x}\hookrightarrow E \quad \text{if} \quad E' \subset E \text{ and } E/E' \cong \mathbb{C}_x
\end{equation}
We recall the basic maps of diagram \eqref{eqn:diagram zk}:
$$
\xymatrix{& \quot_{d,d+1} \ar[ld]_{{p_-}} \ar[d]^{p_C} \ar[rd]^{{p_+}} & \\ \quot_{d} & C & \quot_{d+1}},
$$
and the fact \eqref{eqn:quot proj 1} that the map $p_- \times p_C$ realizes $\quot_{d,d+1}$ as the projectivization $\mathbb{P}_{\quot_d \times C}(\mathcal{E}).$ This shows that $\quot_{d,d+1}$ is a fine moduli space, smooth and projective of dimension $(d+1)r$. In turn, the map $p_{+} : \quot_{d,d+1} \to \quot_{d+1}$ is generically finite of degree $d+1$. Moreover, we have the following.

\begin{lemma}
\label{lemma:projectivizations}
The map $p_+ \times p_C$ realizes $\quot_{d,d+1}$ as the projectivization
\begin{equation}
\label{eqn:quot proj 2}
\mathbb{P}_{\quot_{d+1} \times C}(\mathcal{E}^\vee \otimes \mathcal{K}_C - V^\vee \otimes \mathcal{K}_C)
\end{equation}
with respect to the dual of the first morphism in \eqref{eqn:ses}. 
\end{lemma}

\begin{proof} We prove the claim on the level of closed points, and leave the natural scheme-theoretic generalization as an exercise to the reader (it follows quite closely the argument given for surfaces in \cite[Proposition 2.19]{negut4}). The fiber of $$p_+ \times p_C: \quot_{d, d+1} \to \quot_{d+1} \times \, C$$ above a point $(\{E' \subset V\}, x)$ parametrizes
$$
E' \subset E \subset V \quad \text{s.t.} \quad E/E' \cong \mathbb{C}_x. 
$$
Any such flag corresponds up to dilation to a non-zero element of
$$
\text{Ext}^0 ({\mathbb C}_x, F') = \text{Ker }\Big[\text{Ext}^1(\mathbb{C}_x,E') \rightarrow \text{Ext}^1(\mathbb{C}_x,V)\Big].
$$
Using Serre duality, this is the same as an element of 
$$
\text{Ker }\Big[\text{Hom}(E' \otimes \mathcal{K}_C^{-1}, \mathbb{C}_x)^\vee \rightarrow \text{Hom}(V \otimes \mathcal{K}_C^{-1}, \mathbb{C}_x)^\vee \Big] 
=\text{Ker }\Big[ E'_x \otimes \mathcal{K}_{C,x}^{-1} \rightarrow  V_x \otimes \mathcal{K}_{C,x}^{-1} \Big]
$$
as prescribed by \eqref{eqn:quot proj 2}.

\comment{Equivalently, the fiber parametrizes $$V^\vee \subset E^\vee \subset {E'}^\vee \quad \text{s.t.} \quad {E'}^\vee/E^\vee \cong \mathbb{C}_x,$$
in other words it parametrizes morphisms ${E'}^\vee \to {\mathbb C}_x$ such that the induced composition $V^\vee \to {\mathbb C}_x$ is zero. By definition \eqref{eqn:def proj 2}, it follows that 
$$\quot_{d,d+1} \cong \mathbb{P}_{\quot_{d+1} \times C}(\mathcal{E}^\vee - V^\vee) \cong \mathbb{P}_{\quot_{d+1} \times C}(\mathcal{E}^\vee \otimes \mathcal{K}_C - V^\vee \otimes \mathcal{K}_C).$$}

\end{proof}

\subsection{An important subbundle}
\label{sub:g}

The nested Quot scheme $\quot_{d,d+1}$ is endowed with a tautological line bundle $\mathcal{L}$, whose fibers are naturally given by
$$
\CL|_{E' \stackrel{x}\hookrightarrow E \subset V} = E_x/E'_x
$$
It is easy to see that $\mathcal{L}$ is identified with $\mathcal{O}(1)$ and $\mathcal{O}(-1)$ under the projectivizations \eqref{eqn:quot proj 1} and \eqref{eqn:quot proj 2}, respectively. The obvious universal locally free sheaves $\mathcal{E}$ and $\mathcal{E}'$ on $\quot_{d,d+1} \times C$ fit into a short exact sequence
\begin{equation}
\label{eqn:ses nested}
0 \rightarrow \CE' \rightarrow \CE \rightarrow \mathcal{L} \otimes \mathcal{O}_{\Delta} \rightarrow 0
\end{equation}
where $\Delta \hookrightarrow \quot_{d,d+1} \times \, C$ denotes the divisor which is the pullback of the diagonal $\Delta \subset C \times C$ under the map $$p_C \times 1_C: \quot_{d,d+1} \times \, C \to C \times C.$$ If we restrict the short exact sequence \eqref{eqn:ses nested} to $\Delta$, then a simple Tor computation (essentially equivalent to $\text{Tor}_1^C(\BC_x,\BC_x) \cong \mathfrak{m}_x/\mathfrak{m}_x^2$ for a point $x \in C$) shows that we have an exact sequence
\begin{equation}
\label{eqn:tor nested}
0 \rightarrow \mathcal{L} \otimes \mathcal{K}_C \xrightarrow{\alpha} \CE'_\Delta \rightarrow \CE_\Delta \xrightarrow{\beta} \mathcal{L} \rightarrow 0
\end{equation}
of sheaves on $\quot_{d,d+1}$. In the formula above, we abuse notation by writing $\mathcal{K}_C$ for the pull-back of the canonical line bundle to $\quot_{d,d+1}$ via $p_C$. We note that 
$$
\mathcal{G} := \text{Coker }\alpha = \text{Ker }\beta 
$$
is a rank $r-1$ locally free sheaf on $\quot_{d,d+1}$. As a consequence of \eqref{eqn:tor nested}, we have
\begin{equation}
\label{eqn:chern g}
c(\mathcal{G},z) = \frac {c(\CE_\Delta,z)}{z-\lambda} = \frac {c(\CE'_\Delta,z)}{z-\lambda - K_C}  \in \fH_{\quot_{d,d+1}}[z]
\end{equation}
where we let
\begin{equation}
\label{eqn:lambda}
\lambda = c_1 (\mathcal L) \in \fH^2_{\quot_{d,d+1}} 
\end{equation}
and write $K_C \in \fH^2_C$ for the canonical divisor. By abusively ignoring to write down pull-back maps, $K_C$ is a well-defined cohomology class on any variety which maps to $C$, such as $\quot_{d,d+1}$ via the map $p_C$.

\section{A commuting family of operators}  

\subsection{Overview}

The main purpose of the present Section is to prove Proposition \ref{prop:commute a intro}. Recall the operators introduced in \eqref{eqn:formula a intro}, namely
$$
a_k = (p_+ \times p_C)_* \Big( c_k(\mathcal{G}) \cdot p_-^* \Big) : \fH_{\quot_d} \rightarrow \fH_{\quot_{d+1} \times C},
$$
for any $k \in \{0,\dots,r-1\}$, where the maps $p_\pm$ and $p_C$ are as in \eqref{eqn:diagram zk}. We will show that these operators commute, namely 
\begin{equation}
\label{eqn:commute a}
a_k  a_{k'} = a_{k'} a_k : \fH_{\quot_d} \rightarrow \fH_{\quot_{d+2} \times C \times C},
\end{equation}
thus establishing Proposition \ref{prop:commute a intro}. Here the map $a_k$ (respectively $a_{k'}$) takes values in the first (respectively second) copy of $C \times C$ on both sides of equation \eqref{eqn:commute a}. Note that equality \eqref{eqn:commute a} is non-trivial even for $k = k'$, due to the implicit switch of the two factors of $C \times C$ between the left and right-hand sides. 

Before starting the proof of Proposition \ref{prop:commute a intro}, let us recall its main application. Let $|0\rangle$ denote the fundamental class of $\quot_0 = \text{point}$. As a consequence of \eqref{eqn:commute a}, we note that the classes 
\begin{equation}
\label{eqn:basis of cohomolgy}
\Big\{ a_{k_1} \dots a_{k_d}(\gamma) |0\rangle \in \fH_{\quot_d}\Big\}_{k_1,\dots,k_d \in \{0,\dots,r-1\}, \gamma \in \fH_{C^d}}
\end{equation}
are unchanged by the simultaneous permutation of $k_i$ and $k_j$ and of the $i$-th and $j$-th factors of $C^d$. In Section \ref{sec:theorem 1}, we will show that the classes \eqref{eqn:basis of cohomolgy} for
$$
r > k_1 \geq \dots \geq k_d \geq 0 \qquad \text{and} \qquad \gamma \in \,\, \mathbb Q\, \text{-basis of } \fH_{C^d,\Sigma}
$$
yield a basis of $\fH_{\quot_d}$, where $\Sigma \subset S(d)$ is the subgroup of permutations generated by those transpositions $(ij)$ such that $k_i = k_j$. Therefore, the Poincar\'{e} polynomial of $\quot_d$ is given by enumerating such classes $a_{k_1} \dots a_{k_d}(\gamma)|0\rangle$, and it is straightforward to obtain formula \eqref{eqn:poincare}. Let us note that $a_{k_1} \dots a_{k_d}(\gamma)|0\rangle$ is explicitly given by considering the nested Quot scheme parametrizing
$$
\quot_{0,1,\dots,d} = \Big \{ E_0 \stackrel{x_1}\hookrightarrow E_1 \stackrel{x_2}\hookrightarrow \dots \stackrel{x_d}\hookrightarrow E_d = V \Big\}
$$
and pushing forward $c_{k_1}(\mathcal{G}_1) \dots c_{k_d}(\mathcal{G}_d) \rho^*(\gamma)$ from $\quot_{0,1,\dots,d}$ to $\quot_d$ via the natural forgetful map (here $\rho : \quot_{0,1,\dots,d} \to C^d$ be the map which remembers $(x_1,\dots,x_d)$). 

\subsection{Two-step nested Quot schemes}
\label{sub:two-step}

For the setup of Proposition \ref{prop:commute a intro}, as well as some of the Yangian identities that we will establish in Section \ref{sec:rep theory}, we need to consider the two-step nested Quot scheme
\begin{equation}
\label{eqn:def two-step nested quot scheme}
\quot_{d,d+1,d+2} = \Big\{ E'' \stackrel{y}\hookrightarrow E' \stackrel{x}\hookrightarrow E \subset V  \Big\}
\end{equation}
where the right-most inclusion has colength $d$. 
It comes equipped with nested universal locally free sheaves
$$
{\mathcal E}'' \hookrightarrow \mathcal E' \hookrightarrow \mathcal E \subset \rho^*(V) \, \, \, \text{on} \, \, \, \quot_{d, d+1, d+2} \times \, C
$$
where $\rho : \quot_{d,d+1,d+2} \times C \to C$ is the natural projection. We view the three universal subsheaves above as pulled back under the inclusion
$$
\quot_{d, d+1, d+2} \subset \quot_d \times \quot_{d+1} \times \quot_{d+2}.
$$
We note also the natural morphism $$p_1 \times p_2: \, \quot_{d, d+1, d+2} \to C \times C, \qquad [E'' \stackrel{y}\hookrightarrow E' \stackrel{x}\hookrightarrow E \subset V]\, \, \,  \mapsto \, \, \,  (x, y). $$
There are two associated diagonal divisors $$\Delta_1, \Delta_2 \, \subset \, \quot_{d, d+1, d+2} \times C,$$ obtained by pulling back the diagonal $\Delta \subset C \times C$ under the maps $$p_1 \times 1_C, \, p_2 \times 1_C: \, \quot_{d, d+1, d+2} \, \times \, C \to C \times C.$$
By analogy with Lemma \ref{lemma:projectivizations}, it is clear that 
\begin{equation}
\label{eqn:two step quot as proj}
\quot_{d,d+1,d+2} = {\mathbb P}_{\quot_{d,d+1} \times C}(\mathcal{E}').
\end{equation}
Thus $\quot_{d, d+1, d+2}$ is a double projective bundle over $\quot_d \times C \times C$, in particular a smooth projective variety of dimension $r (d+2).$ 
Thanks to the double projective tower structure, there are two hyperplane line bundles $$\mathcal L_1, \mathcal L_2 \, \, \to \, \, \quot_{d,d+1,d+2},$$parametrizing the quotients $E_x/E'_x$, $E'_y/E''_y$, respectively. It is easy to see that $\mathcal L_2$ coincides with $\mathcal O(1)$ under the identification \eqref{eqn:two step quot as proj}, while $\mathcal L_1$ is pulled back from $\quot_{d,d+1}$.

We note the universal short exact sequences
\begin{equation}
\label{eqn:L12}
 0 \to \mathcal E' \to \mathcal E \to \mathcal L_1 \otimes \mathcal O_{\Delta_1} \to 0 \, \, \, \text{and} \, \, \, 0 \to \mathcal E'' \to \mathcal E' \to \mathcal L_2 \otimes \mathcal O_{\Delta_2} \to 0
 \end{equation}
on $\quot_{d, d+1, d+2} \times \, C,$
and the accompanying short exact sequences
\begin{equation}
\label{eqn:G12}
0 \to \mathcal G_1 \to {\mathcal E}_{\Delta_1} \to \mathcal L_1  \to 0 \, \, \, \text{and} \, \, \, 0 \to \mathcal G_2 \to {\mathcal E}_{\Delta_2}' \to \mathcal L_2  \to 0
\end{equation}
on $\quot_{d, d+1, d+2}.$
From \eqref{eqn:L12} we conclude
\begin{equation}
\label{eqn:chernL12}
c( \mathcal E', z) = c(\mathcal E, z) \cdot \frac{z-\lambda_1 + \Delta_1}{z- \lambda_1} \, \,\, \text{and} \, \, \, c( \mathcal E'', z) = c(\mathcal E', z) \cdot \frac{z-\lambda_2 + \Delta_2}{z- \lambda_2}
\end{equation}
in $\fH_{\quot_{d, d+1, d+2} \times C}$. Using \eqref{eqn:G12}, we can rewrite these equations as
\begin{equation}
\label{eqn:Gdiag}
c( \mathcal E', z) - c(\mathcal E, z) = c(\mathcal G_1, z) \cdot \Delta_1 \, \,\, \text{and} \, \, \, c( \mathcal E'', z) - c(\mathcal E', z) =  c(\mathcal G_2, z) \cdot  \Delta_2.
\end{equation}
Here $c(\mathcal G_1, z), \, c(\mathcal G_2, z)$ denote the pullbacks of these Chern polynomials to $\quot_{d, d+1, d+2} \times C.$ Finally from \eqref{eqn:chernL12} we record for future use the identity
\begin{equation}
\label{eqn:chernL12together}
c( \mathcal E'', z) = c( \mathcal E, z) \cdot \frac{z-\lambda_1 + \Delta_1}{z- \lambda_1} \cdot \frac{z-\lambda_2 + \Delta_2}{z- \lambda_2}
\end{equation}
in  $\fH_{\quot_{d, d+1, d+2} \times C}$.

\subsection{The scheme $\quot_{d, d+2}$}
\label{sub:quot d d+2}

We consider now the nested Quot scheme $$\quot_{d, d+2} = \{ E'' \hookrightarrow E \subset V, \, \, \, \text{with length } V/E = d, \, \, \, \text{length } V/E'' = d+2 \}.$$
As a nested Quot scheme on a nonsingular curve, $\quot_{d,d+2}$ is known to be nonsingular and connected (see for example \cite{mr}, Sections 1 and 2). 
For each point in $\quot_{d, d+2}$, the quotient $T = E/E''$ has length 2. There is accordingly a universal short exact sequence $$0 \to \mathcal E'' \to \mathcal E \to \mathcal T \to 0 \, \, \, \text{on} \, \, \, \quot_{d, d+2} \times \, C,$$ with a universal torsion sheaf $\mathcal T$ of relative length 2.

There is a natural morphism 
$$\tau: \quot_{d, d+1, d+2} \to \quot_{d, d+2}, \qquad  [E'' \stackrel{y}\hookrightarrow E' \stackrel{x}\hookrightarrow E \subset V]\, \, \,  \mapsto \, \, \, [E'' \hookrightarrow E \subset V]$$
which forgets the middle sheaf. We note that $\tau$ is generically $2:1$. 
We may also consider the map
$$\tau \times p_1 \times p_2:  \quot_{d, d+1, d+2} \to \quot_{d, d+2} \times C \times C,$$ $$  [E'' \stackrel{y}\hookrightarrow E' \stackrel{x}\hookrightarrow E \subset V]\, \,   \mapsto \, \,  [E'' \hookrightarrow E \subset V, \, x, \, y],$$
which records in addition the two points at which $E/E'$ and $E'/ E''$ are supported. 

We observe that given a point $[E'' \hookrightarrow E \subset V] \in \quot_{d, d+2}$, specifying a middle sheaf $E'$ is equivalent to specifying a surjective morphism $T \to \mathbb C_x$ for $x \in \text{supp}\, T.$ (Here we set as before $T = E/E''$). Set-theoretically we have therefore  
\begin{equation}
\label{eqn:nestedeq}
\quot_{d, d+1, d+2} \cong \mathbb P_{ \quot_{d, d+2} \times C} (\mathcal E - \mathcal E'')\,  \stackrel{\tau \times p_1}\longrightarrow \, \quot_{d, d+2} \times C,
\end{equation}
where the right-hand side is the projectivization
$\mathbb P_{ \quot_{d, d+2} \times C} (\mathcal E - \mathcal E''),$ with hyperplane line bundle $\mathcal L_1 \to \quot_{d, d+1, d+2}.$ The scheme-theoretic isomorphism \eqref{eqn:nestedeq} holds as well, by a straightforward comparison of the Zariski tangent spaces of the two sides (cf. \cite{mr} Section 2).

The general pushforward formula \eqref{eqn:segre 2} now gives
\begin{equation}
\label{eqn:tau_p1}
(\tau\times p_1)_* \left [ \frac{1}{z-\lambda_1} \right ] = \left[ \frac{c (\mathcal E'', z)}{c (\mathcal E, z)}\right]_{z^{<0}} = \frac{c (\mathcal E'', z)}{c (\mathcal E, z)} -1 
\end{equation}
(since $\text{rank }\CE = \text{rank }\CE''$, the RHS is a power series in $z^{-1}$ with no constant term) under $$\tau \times p_1: \quot_{d, d+1, d+2} \to \quot_{d, d+2} \times \, C.$$
This will help us establish the following.
\begin{lemma}
\label{lemma:pushtau}
Under the map $\tau \times 1_C: \quot_{d, d+1, d+2} \times \, C \rightarrow \quot_{d, d+2} \times\,  C,$ we have 
\begin{equation}
\label{eqn:lemma 2}
(\tau \times 1_C)_* \, c(\mathcal E', z) = c(\mathcal E, z) + c (\mathcal E'', z)
\end{equation}
\end{lemma}
\begin{proof}
We calculate
\begin{eqnarray*}
(\tau \times 1_C)_* \, c(\mathcal E', z) &=& (\tau \times 1_C)_* \left [ c(\mathcal E, z) \left ( 1 + \frac{\Delta_1}{z-\lambda_1} \right ) \right ] \\
& = & c(\mathcal E, z) \cdot (\tau \times 1_C)_* \left [  1 + \frac{\Delta_1}{z-\lambda_1} \right ] \\
& = & c(\mathcal E, z)  \left [  2 + \frac{c (\mathcal E'', z)}{c (\mathcal E, z)} -1 \right ] \\
& = & c(\mathcal E, z) + c (\mathcal E'', z).
\end{eqnarray*}
The third equality above is due to the fact that $\tau$ is generically $2:1$, which implies that 
$$
(\tau \times 1_C)_*(1) = 2,
$$
and the fact that $(\tau \times 1_C)|_{\Delta_1}$ is naturally identified with $\tau \times p_1$, which implies that
$$
(\tau \times 1_C)_* \left(\frac{\Delta_1}{z-\lambda_1} \right) = (\tau \times p_1)_* \left( \frac 1{z-\lambda_1} \right) \stackrel{\eqref{eqn:tau_p1}}= \frac{c (\mathcal E'', z)}{c (\mathcal E, z)} -1.
$$

\end{proof}
Since $\tau$ is generically $2:1$, formula \eqref{eqn:lemma 2} is equivalent to
\begin{equation}
\label{eqn:lemma 2 equiv}
(\tau \times 1_C)_* \, [c(\mathcal E', z) - c (\mathcal E, z) ]= c(\mathcal E'', z) - c (\mathcal E, z).
\end{equation}

\subsection{The proof of commutativity} We are now ready to prove Proposition \ref{prop:commute a intro}.
Let us combine the operators \eqref{eqn:formula a intro} into a generating polynomial 
\begin{equation}
\label{eqn:formula a polynomial}
a(z) = \sum_{k=0}^{r-1} z^{r-1-k} (-1)^k a_k : \fH_{\quot_d} \rightarrow \fH_{\quot_{d+1} \times C}[z]
\end{equation}
With the notation as in Subsection \ref{sub:g}, we have 
\begin{equation}
\label{eqn:def a}
a(z) = (p_+ \times p_C)_* \Big(c(\mathcal{G},z) \cdot p_-^*\Big) 
\end{equation}

\begin{proof} \emph{of Proposition \ref{prop:commute a intro}:} We need to prove the identity
\begin{equation}
\label{eqn:a series commute}
a(z) a(w) = a(w) a(z),
\end{equation}
where on both sides of the equation, the operators making up the $w$-series (respectively the $z$-series) take values in the first (respectively second) factor of the product $C \times C$. 
We establish \eqref{eqn:a series commute} as an equality of correspondences on $\quot_{d, d+2} \times C \times C,$ by showing
$$
(\tau \times p_1 \times p_2)_* \left ( c (\mathcal G_1, w) \cdot c (\mathcal G_2, z) \right ) =  (\tau \times p_2 \times p_1)_* \left (c (\mathcal G_1, z) \cdot c (\mathcal G_2, w) \right ).
$$
As an artifice, we may write  
$$(\tau \times p_1 \times p_2)_* ( c (\mathcal G_1, w) \cdot c (\mathcal G_2, z)) = (\tau \times 1_C \times 1_C)_* ( c(\mathcal G_1, w) \cdot c (\mathcal G_2, z) \cdot  \Delta_1 \cdot \Delta_2 ),$$
where $$\tau \times 1_C \times 1_C: \quot_{d,d+1, d+2} \, \times \, C \times \, C \to \quot_{d, d+2} \times C \times C,$$
and $\Delta_1$ (respectively $\Delta_2$) is the divisor in $\quot_{d,d+1,d+2} \times C \times C$ obtained by identifying the image of $p_1$ (respectively $p_2$) with the first (respectively second) extraneous copy of $C$. This enables us to use \eqref{eqn:Gdiag} to write
$$
(\tau \times p_1 \times p_2)_* ( c(\mathcal G_1, w) \cdot c (\mathcal G_2, z) ) = (\tau \times 1_C \times 1_C)_* \Big [ (c({\mathcal E}', w) -  c(\mathcal E, w) ) \cdot  (c ({\mathcal E}'', z ) -  c (\mathcal E', z))   \Big ]
$$
Above and throughout the argument, the universal sheaves $\CE,\CE',\CE''$ appearing in all $w$-Chern series (respectively $z$-Chern series) correspond to the first (respectively the second) factor of $C$ in the product $\quot_{d,d+1,d+2} \times C \times C$.
We have analogously
$$(\tau \times p_2 \times p_1)_* \left (c (\mathcal G_1, z) \cdot c (\mathcal G_2, w) \right )  = (\tau \times 1_C \times 1_C)_* \Big [  \left (c (\mathcal E', z) -  c (\mathcal E, z) \right ) \cdot  \left (c (\mathcal E'', w ) -  c (\mathcal E', w)\right ) \Big ]
$$
We therefore evaluate the difference 
$$(\tau \times p_1 \times p_2)_* ( c(\mathcal G_1, w) \cdot c (\mathcal G_2, z) ) - (\tau \times p_2 \times p_1)_* \left (c (\mathcal G_1, z) \cdot c (\mathcal G_2, w) \right )$$ to be 
$$(\tau \times 1_C\times 1_C)_* \Big [ (c (\mathcal E', w) - c (\mathcal E, w) )  (c (\mathcal E'', z) -  c (\mathcal E', z)) -  (c (\mathcal E', z) -  c (\mathcal E, z) ) (c (\mathcal E'', w ) -  c (\mathcal E', w)  ) \Big ] 
$$
\begin{multline*} 
=  (\tau \times 1_C \times 1_C)_* \Big [  \left (c (\mathcal E', w) -  c  (\mathcal E, w)\right ) \cdot  c  ({\mathcal E'', z}) + c (\mathcal E, w)\cdot c  (\mathcal E', z) -  \\  -  \left (c (\mathcal E', z) -  c (\mathcal E, z)\right ) \cdot  c (\mathcal E'', w ) - c (\mathcal E, z) \cdot c (\mathcal E', w) \Big ].
\end{multline*}
Invoking formulas \eqref{eqn:lemma 2} and \eqref{eqn:lemma 2 equiv}, the quantity above equals
\begin{align*}
&( c (\mathcal E'', w) -  c (\mathcal E, w)) \cdot  c (\mathcal E'', z ) + c (\mathcal E, w)\cdot (c  (\mathcal E'', z) + c (\mathcal E, z)) -\\
&-(c (\mathcal E'', z)-  c (\mathcal E, z)) \cdot  c (\mathcal E'', w ) -  c (\mathcal E, z) \cdot ( c (\mathcal E'', w ) + c (\mathcal E, w) )= 0 
\end{align*}
as required.

\end{proof}

\section{The Yangian action}
\label{sec:rep theory}

\subsection{The commutation relations}

Recall that we combined the creation/annihilation operators of \eqref{eqn:formula e intro}--\eqref{eqn:formula f intro} into power series
\begin{align}
&e(z) = \sum_{k = 0}^{\infty} \frac {e_k}{z^{k+1}} = (p_+ \times p_C)_* \left[\frac 1{z-\lambda} \cdot p_-^* \right] \label{eqn:operator e} : \fH_{\quot_d} \rightarrow \fH_{\quot_{d+1} \times C}[[z^{-1}]] \\
&f(z) = \sum_{k = 0}^{\infty} \frac {f_k}{z^{k+1}} = (p_- \times p_C)_* \left[\frac 1{z-\lambda} \cdot p_+^* \right] : \fH_{\quot_{d+1}} \rightarrow \fH_{\quot_{d} \times C}[[z^{-1}]] \label{eqn:operator f} 
\end{align}
where $\lambda = c_1(\mathcal{L}) \in \fH^2_{\quot_{d,d+1}}$. Similarly, we combined the multiplication operators \eqref{eqn:operator m intro} into the following degree $r$ monic polynomial in $z$
\begin{equation}
\label{eqn:operator m}
m(z) =\text{multiplication by } c(\mathcal{E},z)  : \fH_{\quot_d} \rightarrow \fH_{\quot_{d} \times C}[z]
\end{equation}
Recall the following power series in $z^{-1}$
\begin{equation}
\label{eqn:operator h}
h(z) = \frac {c(V,z+K_C)}{m(z)m(z+K_C)} = \text{multiplication by } \frac {c(V,z+K_C)}{c(\mathcal{E},z)c(\mathcal{E},z+K_C)} 
\end{equation}
When $V$ is a vector bundle on a curve, the numerator $c(V,z)$ in ordinary cohomology is equal to $z^r - z^{r-1} c_1(V)$. However, everything that will follow also holds in equivariant cohomology, in which case $c(V,z)$ may be an arbitrary degree $r$ polynomial. 

\begin{proof} \emph{of Theorem \ref{theorem:yangian}:} We follow ideas developed in \cite{negut2, negut4} in the context of moduli spaces of sheaves on surfaces. Relation \eqref{eqn:rel mm_quot} is obvious, as all multiplication operators commute. Let us first prove \eqref{eqn:rel ee_quot}, namely
$$
\Big[ e (z) e (w) ( z- w + \delta) \Big]_{z^{<0},w^{<0}} = \Big[ e (w) e (z) (z-w-\delta) \Big]_{z^{<0},w^{<0}}.
$$
Just as in the case of the $a$-operators, we agree that on both sides of this equation, the operators making up the $w$-series (respectively the $z$-series) take values in the first (respectively second) factor of the product $C \times C$. 
We will establish the identity displayed above as an equality of correspondences on $\quot_{d, d+2} \times C \times C.$ 
To start, we note that
\begin{equation}
\label{eqn:e e as a correspondence}
e (z) e (w) ( z- w + \delta) = (z - w + \delta) \cdot  (\tau \times p_1 \times p_2)_* \left [ \frac{1}{w- \lambda_1} \cdot \frac{1}{z-\lambda_2} \right ]
\end{equation}
with the notation as in Subsection \ref{sub:quot d d+2}. It will be more convenient to use the morphism $$\tau \times p_1 \times 1_C: \, \quot_{d, d+1, d+2} \, \times \, C \to \quot_{d, d+2} \times C \times C,$$ noting that 
for any class $\gamma \in \fH_{\quot_{d, d+1, d+2}}$, we have 
$$ (\tau \times p_1 \times p_2)_*(\gamma) = (\tau \times p_1 \times 1_C)_* (\gamma \cdot \Delta_2),$$where on the right hand side $\gamma$ is pulled back to the product $\quot_{d, d+1, d+2} \times \, C.$
We are thus able to rewrite \eqref{eqn:e e as a correspondence} as 
$$
e (z) e (w) ( z- w + \delta) = (z - w + \delta) \cdot (\tau \times p_1 \times 1_C)_* \left [ \frac{1}{w- \lambda_1} \cdot \frac{\Delta_2}{z-\lambda_2}  \right ]
$$
From \eqref{eqn:chernL12together}, we now express
$$\frac{\Delta_2}{z - \lambda_2} = \frac{c ({\mathcal E}'', z)}{c({\mathcal E}, z)} \cdot \frac{z-\lambda_1} {z-\lambda_1 + \Delta_1} -1 \, \, \text{on} \, \, \quot_{d, d+1, d+2} \times C,$$
and write 
$$
e (z) e (w) ( z- w + \delta) = (z - w + \delta) \cdot (\tau \times p_1 \times 1_C)_* \left [  \frac{1}{w- \lambda_1} \cdot \left [\frac{c ({\mathcal E}'', z)}{c({\mathcal E}, z)} \cdot \frac{z-\lambda_1} {z-\lambda_1 + \Delta_1} -1 \right ] \right ]$$
$$= (z - w + \delta) \cdot (\tau \times p_1 \times 1_C)_* \left [ \frac{1}{w- \lambda_1} \cdot \left [\left (\frac{c ({\mathcal E}'', z)}{c({\mathcal E}, z)}-1 \right ) - \frac{c ({\mathcal E}'', z)}{c({\mathcal E}, z)} \cdot  \frac{\Delta_1} {z-\lambda_1 + \Delta_1} \right ] \right ].
$$
To evaluate this pushforward expression on $\quot_{d, d+2} \times C \times C, $ we make essential use of \eqref{eqn:tau_p1}. We also observe that the divisor $\Delta_1 \in \fH^2_{\quot_{d, d+1, d+2} \times C}$ is the pullback of the diagonal class $\delta \in \fH^2_{\quot_{d, d+2} \times C \times C}$ under $\tau \times p_1 \times 1_C.$ Finally, the universal sheaves $\mathcal E, \, \mathcal E''$ in the $w$-series (respectively $z$-series) are always associated with the first (respectively second) curve factor in the product $\quot_{d, d+2} \times C \times C$; we leave this out of the notation for simplicity. We thus obtain that $e (z) e (w) ( z- w + \delta)$ is equal to
$$ (z-w + \delta) \left [\frac{c ({\mathcal E}'', w)}{c({\mathcal E}, w)} - 1 \right ] \left [\frac{c ({\mathcal E}'', z)}{c({\mathcal E}, z)} - 1 \right ] - \, \delta \, \cdot \, \frac{c ({\mathcal E}'', z)}{c({\mathcal E}, z)}  (\tau \times p_1 \times 1_C)_* \frac{z-w + \Delta_1}{(w-\lambda_1)(z-\lambda_1 + \Delta_1)}$$
$$= (z-w + \delta) \left [\frac{c ({\mathcal E}'', w)}{c({\mathcal E}, w)} - 1 \right ] \left [\frac{c ({\mathcal E}'', z)}{c({\mathcal E}, z)} - 1 \right ] - \, \delta \, \cdot \, \frac{c ({\mathcal E}'', z)}{c({\mathcal E}, z)}  (\tau \times p_1 \times 1_C)_* \left [ \frac{1}{w-\lambda_1} - \frac{1}{z- \lambda_1 + \Delta_1} \right ]$$
$$= (z-w + \delta)  \left [\frac{c ({\mathcal E}'', w)}{c({\mathcal E}, w)} - 1 \right ]  \left [\frac{c ({\mathcal E}'', z)}{c({\mathcal E}, z)} - 1 \right ] - \, \delta \, \cdot \, \frac{c ({\mathcal E}'', z)}{c({\mathcal E}, z)}  \left [ \frac{c ({\mathcal E}'', w)}{c({\mathcal E}, w)} -  \frac{c ({\mathcal E}'', z + \delta)}{c({\mathcal E}, z + \delta)}\right ]$$ 
on $\quot_{d, d+2} \times C \times C.$ 

Turning now to the right-hand side $e (w) e (z) ( z- w - \delta),$ we write analogously
\begin{eqnarray*}
e (w) e (z) ( z- w - \delta) &=& (z - w - \delta) \cdot (\tau \times p_2 \times p_1)_* \left [ \frac{1}{z- \lambda_1} \cdot \frac{1}{w-\lambda_2}  \right ]\\
& = &(z - w - \delta) \cdot (\tau \times 1_C \times p_1)_* \left [ \frac{1}{z- \lambda_1} \cdot \frac{\Delta_2}{w-\lambda_2}  \right ].
\end{eqnarray*}
We obtain
$$e (w) e (z) ( z- w - \delta) = (z - w - \delta) \cdot (\tau \times \, 1_C \, \times \, p_1)_* \left [  \frac{1}{z- \lambda_1} \cdot \left [\frac{c ({\mathcal E}'', w)}{c({\mathcal E}, w)} \cdot \frac{w-\lambda_1} {w-\lambda_1 + \Delta_1} -1 \right ] \right ]$$
$$ = (z-w - \delta) \left [\frac{c ({\mathcal E}'', z)}{c({\mathcal E}, z)} - 1 \right ] \left [\frac{c ({\mathcal E}'', w)}{c({\mathcal E}, w)} - 1 \right ] - \, \delta \, \cdot \, \frac{c ({\mathcal E}'', w)}{c({\mathcal E}, w)}  (\tau \times 1_C \times p_1)_* \frac{z-w - \Delta_1}{(z- \lambda_1)(w- \lambda _1 + \Delta_1)}$$
$$ = (z-w - \delta) \left [\frac{c ({\mathcal E}'', z)}{c({\mathcal E}, z)} - 1 \right ] \left [\frac{c ({\mathcal E}'', w)}{c({\mathcal E}, w)} - 1 \right ] - \, \delta \, \cdot \, \frac{c ({\mathcal E}'', w)}{c({\mathcal E}, w)}  (\tau \times 1_C \times p_1)_* \left [ \frac{1}{w - \lambda_1 +\Delta_1} - \frac{1}{z- \lambda _1 } \right ]$$
$$ = (z-w - \delta) \left [\frac{c ({\mathcal E}'', z)}{c({\mathcal E}, z)} - 1 \right ] \left [\frac{c ({\mathcal E}'', w)}{c({\mathcal E}, w)} - 1 \right ] - \, \delta \, \cdot \, \frac{c ({\mathcal E}'', w)}{c({\mathcal E}, w)}  \left [ \frac{c ({\mathcal E}'', w +\delta)}{c({\mathcal E}, w + \delta)} -  \frac{c ({\mathcal E}'', z)}{c({\mathcal E}, z)}\right ].$$
The expression above is a cohomology class on $\quot_{d, d+2} \times C \times C$; as before, we recall that the universal sheaves $\mathcal E, \mathcal E''$ appearing in the $w$-Chern series (respectively $z$-series) are considered on the first (respectively second) $C$ factor. The difference of the two series $e(z) e(w) (z-w + \delta) - e(w) e (z) (z-w - \delta)$ is now easily seen to be 
\begin{equation}
\label{eqn:nozw}
 \delta \cdot [ g(z) + g (w)] \in \fH_{\quot_{d, d+2} \times\,  C \times \, C}
\end{equation}  
where 
$$g (z) = 1  - 2 \, \frac{c ({\mathcal E}'', z)}{c({\mathcal E}, z)}  + \frac{c ({\mathcal E}'', z)}{c({\mathcal E}, z)} \cdot \frac{c ({\mathcal E}'', z + \delta)}{c({\mathcal E}, z + \delta)}.$$
Since \eqref{eqn:nozw} contains no mixed terms $z^i w^j$ with  $i, j < 0,$ we have established \eqref{eqn:rel ee_quot}. Relation \eqref{eqn:rel ff_quot} is proved identically, as it involves the same correspondences. 

We next prove \eqref{eqn:rel me_quot}, namely the equality
\begin{equation}
\label{eqn:m and e}
m(w)e(z) = \left[ e(z)m(w) \frac {w-z+\delta}{w-z} \right]_{w \gg z|z^{<0}, w^{\geq 0}}
\end{equation}
of power series of operators $\fH_{\quot_d} \rightarrow \fH_{\quot_{d+1} \times C \times C}[[z^{-1}]][w]$. The symbol $[\dots]_{w \gg z|z^{<0},w^{\geq 0}}$ means that we expand the RHS as a power series in $w \gg z$, and from the result only retain the negative powers of $z$ and the non-negative powers of $w$. The short exact sequence \eqref{eqn:ses nested} on $\quot_{d,d+1} \times C$ implies the equality of Chern polynomials
$$
c(\mathcal{E}',w) = c(\mathcal{E},w) \cdot \frac {w-\lambda+\delta}{w-\lambda} \in \fH_{\quot_{d,d+1} \times C}
$$
The formula above will allow us to compute $m(w) e(z)$ and $e(z) m(w)$ as correspondences in $\fH_{\quot_{d,d+1} \times C}$, with the latter factor of $C$ corresponding to the operator $m(w)$. Specifically,
\begin{align*}
&m(w) e(z) \quad \text{is given by the correspondence} \quad \frac {c(\mathcal{E}',w)}{z-\lambda} = \frac {c(\mathcal{E},w)}{z-\lambda} \cdot \frac {w-\lambda+\delta}{w-\lambda} \\
&e(z)m(w) \quad \text{is given by the correspondence} \quad \frac {c(\mathcal{E},w)}{z-\lambda}
\end{align*}
Thus, we have the following equality of correspondences
\begin{multline}
\label{eqn:mult correspondences}
m(w) e(z) - e(z)m(w) \frac {w-z+\delta}{w-z} = \\
= \frac {c(\mathcal{E},w)}{z-\lambda}  \left(\frac {w-\lambda+\delta}{w-\lambda} - \frac {w-z+\delta}{w-z} \right) =  \frac {-\delta \cdot c(\mathcal{E},w)}{(w-\lambda)(w-z)} = \Delta_* \left( \frac {- c(\mathcal{G},w)}{w-z} \right)
\end{multline}
(recall that $\Delta \cong \quot_{d,d+1} \hookrightarrow \quot_{d,d+1} \times C$ is the graph of the map $p_C$). Since $\CG$ is a locally free sheaf of rank $r-1$, the right-hand side of \eqref{eqn:mult correspondences} does not have any terms in the expansion as $w \gg z$ with negative powers of $z$. We therefore conclude \eqref{eqn:m and e}. Relation \eqref{eqn:rel fm_quot} is proved just like the preceding discussion, so we leave it as an exercise to the reader. 

Finally, let us prove \eqref{eqn:rel ef_quot}, specifically the equality
$$
e(z)f(w) - f(w)e(z) = \delta \cdot \frac {h(z)-h(w)}{z-w}
$$
of operators $\fH_{\quot_d} \rightarrow \fH_{\quot_d \times C \times C}[[z^{-1},w^{-1}]]$. As explained in Subsection \ref{sub:corr}, the compositions $e(z)f(w)$ and $f(w)e(z)$ are given by correspondences supported on the moduli spaces 
$$
\{E_1 \stackrel{x}\hookrightarrow E \stackrel{y}\hookleftarrow E_2\} \qquad \text{and} \qquad \{E_1 \stackrel{y}\hookleftarrow E' \stackrel{x}\hookrightarrow E_2\}
$$
respectively (where all coherent sheaves above are viewed as subsheaves of $V$). The assignments 
\begin{align*}
&E \mapsto E_1 \cap E_2  \\
&E' \mapsto E_1 + E_2 
\end{align*}
(as subsheaves of $V$) give mutually inverse maps between the aforementioned correspondences away from the locus $(E_1,x) \neq (E_2,y)$. The excision property of cohomology therefore implies that the commutator $e(z)f(w) - f(w)e(z)$ is supported on the diagonal $\{(E_1,x) = (E_2,y)\}$ of $(\quot_d \times C)^{\times 2}$. We conclude that
\begin{equation}
\label{eqn:commutator general}
e(z)f(w) - f(w)e(z) = \text{multiplication by }\Psi(z,w)
\end{equation}
for some $\Psi(z,w) \in \fH_{\quot_d \times C} [[z^{-1},w^{-1}]]$. To compute the class $\Psi(z,w)$, it suffices to apply relation \eqref{eqn:commutator general} to the fundamental class
$$
e(z)f(w) \cdot 1 - f(w) e(z) \cdot 1 = \Psi(z,w)
$$
Formulas \eqref{eqn:segre 1} and \eqref{eqn:segre 2} applied to the setting of \eqref{eqn:quot proj 1} and \eqref{eqn:quot proj 2} (respectively) give
\begin{align}
&f(w) \cdot 1 = \frac 1{c(\mathcal{E},w)} \label{eqn:f on 1}  \\
&e(z) \cdot 1 = \left[- \frac {c(V,z+K_C)}{c(\mathcal{E},z+K_C)} \right]_{z^{<0}} = 1- \frac {c(V,z+K_C)}{c(\mathcal{E},z+K_C)}
\label{eqn:e on 1}
\end{align}
(the minus sign is due to $\mathcal{L}$ being isomorphic to $\CO(-1)$ on the projectivization \eqref{eqn:quot proj 2}). Let us apply $f(w)$ to the class \eqref{eqn:e on 1}. 
With the notation in Subsection \ref{sub:nested}, we have
\begin{eqnarray*}
\label{eqn:temp}
f(w) e(z) \cdot 1 &=& f(w) \cdot \left[1 - \frac {c(V,z+K_C)}{c(\mathcal{E}',z+K_C)} \right] \\ &=&  (p_- \times p_C)_* \left [\frac 1{w-\lambda} \cdot \left ( 1- \frac {c(V,z+K_C)}{c(\mathcal{E}',z+K_C)} \right ) \right ].
\end{eqnarray*}
 The short exact sequence \eqref{eqn:ses nested} yields
$$
c(\mathcal E',z) = c(\mathcal E,z) \cdot \frac {z-\lambda+\delta} {z-\lambda}
$$
where $\delta$ denotes the class of the diagonal in $C \times C$. We therefore write 
$$
f(w) e(z) \cdot 1 = (p_- \times p_C)_* \left [ \frac{1}{w-\lambda} \left [1- \frac {c(V,z+K_C)}{c(\mathcal{E},z+K_C)} \cdot \frac {z-\lambda +K_C}{z-\lambda +K_C+\delta}\right] \right] $$
$$ =  (p_- \times p_C)_* \left [ \frac{1}{w-\lambda} \cdot \left ( 1- \frac {c(V,z+K_C)}{c(\mathcal{E},z+K_C)} \right ) + \delta \cdot \frac {c(V,z+K_C)}{c(\mathcal{E},z+K_C)} \cdot \frac {1}{(w-\lambda )(z-\lambda +K_C+\delta)}\right].
$$
Since $\delta \cdot (\delta + K_C) = 0 \in \fH_{C \times C}$, the expression above simplifies to 
$$ f(w) e(z) \cdot 1=  (p_- \times p_C)_* \left [ \frac{1}{w-\lambda} \cdot \left ( 1- \frac {c(V,z+K_C)}{c(\mathcal{E},z+K_C)} \right ) + \delta \cdot \frac {c(V,z+K_C)}{c(\mathcal{E},z+K_C)} \cdot \frac {1}{(w-\lambda )(z-\lambda)}\right]$$
$$=  (p_- \times p_C)_* \left [ \frac{1}{w-\lambda} \cdot \left ( 1- \frac {c(V,z+K_C)}{c(\mathcal{E},z+K_C)} \right ) + \frac{\delta}{w-z}  \cdot \frac {c(V,z+K_C)}{c(\mathcal{E},z+K_C)}  \cdot  \left(  \frac{1}{z-\lambda} - \frac{1}{w-\lambda} \right )   \right] $$
Now the pushforward formula \eqref{eqn:segre 1} applied to the setting of \eqref{eqn:quot proj 1} allows us to evaluate
\begin{multline}
\label{eqn:fe}
 f(w) e(z) \cdot 1 = \frac{1}{c(\mathcal E, w)} \cdot \left [ 1- \frac {c(V,z+K_C)}{c(\mathcal{E},z+K_C)} \right ]+ \\ + \frac{\delta}{w-z}  \cdot \frac {c(V,z+K_C)}{c(\mathcal{E},z+K_C)}   \cdot \left [  \frac{1}{c (\mathcal E, z)} - \frac{1}{c(\mathcal E, w)} \right ].
 \end{multline}
We now turn to the composition $e(z) f (w)$ and calculate
\begin{eqnarray*}
e(z) f(w) \cdot 1 &=& (p_{+} \times p_C)_* \left [ \frac{1}{z-\lambda} \cdot \frac{1}{c (\mathcal E_{-}, w)} \right ] \\
&= &(p_{+} \times p_C)_* \left [ \frac{1}{z-\lambda} \cdot \frac{1}{c (\mathcal E, w)} \frac{ w- \lambda + \delta}{w- \lambda} \right ]\\
& = & (p_{+} \times p_C)_* \left [ \frac{1}{z-\lambda} \cdot \frac{1}{c (\mathcal E, w)} +\frac{\delta}{(z-\lambda) (w-\lambda)} \cdot \frac{1}{c (\mathcal E, w)} \right ]\\
& = & (p_{+} \times p_C)_* \left [ \frac{1}{z-\lambda} \cdot \frac{1}{c (\mathcal E, w)} + \frac{\delta}{w-z} \cdot \frac{1}{c (\mathcal E, w)} \cdot \left ( \frac{1}{z-\lambda} - \frac{1}{w-\lambda} \right ) \right].
\end{eqnarray*}
On the first line of the above calculation, $\mathcal E_{-} $ denotes the universal subsheaf on $\quot_{d-1} \times C$ which in the following line we express, on the nested Quot scheme $\quot_{d-1, d},$ in terms of the universal subsheaf $\mathcal E$ on $\quot_d \times C.$ Using the pushforward formula \eqref{eqn:segre 2} in the setting of \eqref{eqn:quot proj 2}, we obtain
\begin{multline}
\label{eqn:ef}
e(z) f(w) \cdot 1 = \left [ 1 - \frac{c(V, z+ K_C)}{c(\mathcal E, z + K_C)} \right ] \cdot \frac{1}{c (\mathcal E, w)} + \\ + \frac{\delta}{w-z} \cdot \frac{1}{c (\mathcal E, w)} \cdot  \left [  \frac{c(V, w+ K_C)}{c(\mathcal E, w + K_C)} - \frac{c(V, z+ K_C)}{c(\mathcal E, z + K_C)} \right ].
\end{multline}
From \eqref{eqn:fe} and \eqref{eqn:ef} we now see that 
$$e(z) f(w) \cdot 1 - f(w) e(z) \cdot 1 = \delta \cdot \frac{h (z) - h (w)}{z-w}, \, \, \text{with} \, \, h (z) = \frac{c(V, z+ K_C)}{c(\mathcal E, z) c(\mathcal E, z + K_C)},$$ as wished. 

\end{proof}

\subsection{The Yangian action} Before we compare the geometric commutation relations of Theorem \ref{theorem:yangian} with the abstract construction of the shifted Yangian in Definition \ref{def:yangian}, we need to give a precise notion of the action of $Y_{\hbar}^r(\mathfrak{sl}_2)$ on $\mathsf H_{\quot} = \bigoplus_{d=0}^{\infty} \fH_{\quot_d}$. Recall that we write
$$
{\mathsf H}_{\quot \times C^n} = \bigoplus_{d=0}^{\infty} \fH_{\quot_d \times C^n}
$$
for any $n \in \BN$.

\begin{definition}
\label{def:action}

An action of $Y_{\hbar}^r(\mathfrak{sl}_2)$ on $\mathsf H_{\quot}$ is a $\BQ$-linear map
\begin{equation}
\label{eqn:action map}
Y_{\hbar}^r(\mathfrak{sl}_2) \xrightarrow{\Phi} \text{Hom} (\mathsf H_{\quot}, \mathsf H_{\quot \times C})
\end{equation}
satisfying the following properties for all $x,y \in Y_{\hbar}^r(\mathfrak{sl}_2)$:

\begin{enumerate}
	
\item $\Phi(1) = \pi^*$, where $\pi : \quot \times C \rightarrow \quot$ is the natural projection map \\

\item $\Phi(\hbar x)$ coincides with the composition:
\begin{equation}
\label{eqn:const}
\mathsf H_{\quot}  \xrightarrow{\Phi(x)} \mathsf H_{\quot \times C} \xrightarrow{\text{Id}_{\quot} \times \emph{multiplication by } K_C}  \mathsf H_{\quot \times C} 
\end{equation}

\item $\Phi(xy)$ coincides with the composition:
\begin{equation}
\label{eqn:hom}
\mathsf H_{\quot} \xrightarrow{\Phi(y)} \mathsf H_{\quot \times C} \xrightarrow{\Phi(x) \times \text{Id}_C} \mathsf H_{\quot \times C \times C} \xrightarrow{\text{Id}_{\quot} \times \Delta^*} \mathsf H_{\quot \times C}
\end{equation}

\item $\Delta_* \Phi \left( \frac {[x,y]}{-\hbar} \right)$ coincides with the difference between the compositions:
\begin{align*}
&\mathsf H_{\quot}\xrightarrow{\Phi(y)} \mathsf H_{\quot \times C} \xrightarrow{\Phi(x) \times \text{Id}_C} \mathsf H_{\quot \times C \times C} \\
&\mathsf H_{\quot} \xrightarrow{\Phi(x)} \mathsf H_{\quot \times C} \xrightarrow{\Phi(y) \times \text{Id}_C} \mathsf H_{\quot \times C \times C} \xrightarrow{\text{Id}_{\quot} \times \text{swap}^*} \mathsf H_{\quot \times C \times C}
\end{align*}
where $\emph{swap} : C \times C \rightarrow C \times C$ is the permutation map. 

\end{enumerate}

Note that for item (4) to be well-defined, we need $[x,y]$ to be a multiple of $\hbar$ for any $x,y \in Y^r_{\hbar}(\mathfrak{sl}_2)$, which is easily seen to be the case from relations \eqref{eqn:rel mm}--\eqref{eqn:rel ef}. 

\end{definition}

We discuss the comparison between the notion above and the more classical notions of actions from \cite{grojnowski}, \cite{nakajima} in Subsection \ref{sub:comparison}. It is clear that Theorem \ref{theorem:yangian} yields an action 
$$
Y_{\hbar}^r(\mathfrak{sl}_2) \xrightarrow{\Phi} \text{Hom} (\mathsf H_{\quot}, \mathsf H_{\quot \times C})
$$
in the sense of Definition \ref{def:action}, thus establishing Corollary \ref{cor:yangian}.

\begin{remark} (Laumon spaces) Let us recall the smooth quasi-projective variety
\begin{equation}
\label{eqn:laumon}
\text{Laumon}_d = \Big \{ \CE \stackrel{\iota}\hookrightarrow \CO_{\BP^1}^{\oplus r}, \text{rank } \CE = r, \deg \CE = -d, \iota \text{ isomorphism near } \infty \in \BP^1\Big\}
\end{equation}
The action $\BC^* \curvearrowright \BP^1$ that fixes 0 and $\infty$ acts on $\text{Laumon}_d$, so we may consider the equivariant cohomology groups
$$
H_{\text{Laumon}} = \bigoplus_{d=0}^{\infty} H^*_{\BC^*}(\text{Laumon}_d)
$$
We identify the abstract symbol $-\hbar$ with the equivariant parameter of the $\BC^*$ action. In \cite[Section 4]{bffr}, it was shown that there exists an action 
\begin{equation}
\label{eqn:yangian on laumon}
Y_{\hbar}^r(\mathfrak{sl}_2) \curvearrowright H_{\text{Laumon}}
\end{equation}
with the generators $e_k,f_k,m_i$ acting by certain correspondences, akin to our \eqref{eqn:operator e}, \eqref{eqn:operator f}, \eqref{eqn:operator m}. $\text{Laumon}_d$ can be construed as an analogue of our $\quot_d$ when the projective curve $C$ is replaced by equivariant $\BA^1$. This case is simpler than that of projective $C$ due to the triviality of the cohomology of $\BA^1$. Thus, for projective $C$ we do not get an honest action of the shifted Yangian on the cohomology of Quot schemes, but instead we obtain the geometric notion in Definition \ref{def:action} (which is inspired by the case of surfaces treated in \cite[Definition 5.2]{negut3}).
\end{remark}

\begin{remark}
\label{rem:ae}

Because of formulas \eqref{eqn:chern g}, it is clear that we have
\begin{equation}
\label{eqn:formula a in terms of e}
a(z) = e(z) m(z) \Big|_\Delta = m(z) e(z-K_C) \Big|_\Delta
\end{equation}
where $e(z)m(z)$ and $m(z)e(z-K_C)$ are operators $\fH_{\quot_d} \rightarrow \fH_{\quot_{d+1} \times C \times C}((z^{-1}))$, and $|_\Delta$ denotes restriction to the diagonal of $C \times C$. Note that with this terminology, formula \eqref{eqn:mult correspondences} reads
\begin{equation}
\label{eqn:m e g}
m(w) e(z) - e(z)m(w) \frac {w-z+\delta}{w-z} = - \delta \cdot \frac {a(w)}{w-z} 
\end{equation}
which is a more ``precise" version of the identity \eqref{eqn:rel me_quot}. 
\end{remark}

\section{Multiplication formulas and the proof of Theorem \ref{theorem:fock basis}} 
\label{sec:theorem 1} 

\subsection{Another nested Quot scheme} In this section, we will express the operators of multiplication $m(z)$ in terms of the creation and annihilation operators, thus obtaining an analogue of Lehn's formula for Hilbert schemes of points on surfaces \cite{lehn}. Consider the scheme
\begin{equation}
\label{eqn:def z}
\mathfrak{Z}_d = \Big\{ E_1 \stackrel{x}\hookrightarrow E \stackrel{x}\hookleftarrow E_2 \ \Big| \ E \text{ is a colength }d -1 \text{ subsheaf of }V, \ x \in C \text{ arbitrary}\Big\} 
\end{equation}
It is easy to see 
that
$$
\mathfrak{Z}_d = \mathbb{P}(\mathcal{E}) \times_{\text{Quot}_{d-1} \,\times \, C}\mathbb{P}(\mathcal{E}).
$$
We note therefore that $\mathfrak{Z}_d$ carries two tautological quotient line bundles $\mathcal{L}_1$ and $\mathcal{L}_2$ on the two projective factors. 
We have a commutative diagram of maps 
\begin{equation}
\label{eqn:comm diag}
\xymatrixcolsep{5pc}
\xymatrix{ \text{Quot}_{d-1,d} \ar[d]_{p_+ \times p_C} \ar[r]^-{\iota} & \mathfrak{Z}_d \ar[d]^{\pi} \\ \text{Quot}_{d} \times C  \ar[r]^-{\Delta_{\text{Quot}} \times \text{Id}_C} & \text{Quot}_{d} \times \text{Quot}_{d} \times C}
\end{equation}
where $\iota$ denotes the closed embedding corresponding to $E_1 = E_2 \subset V,$ and $\pi$ is the map which remembers $(E_1,E_2,x)$ in the notation of \eqref{eqn:def z}. We record the following

\begin{lemma}
\label{lem:diag}

The map $\iota$ is a regular embedding, cut out by the composition of the maps
\begin{equation}
    \label{eqn:sec diag}
\mathcal{G}_1 \hookrightarrow \mathcal{E} \twoheadrightarrow \mathcal{L}_2
\end{equation}
on $\mathfrak{Z}_d$, where $\mathcal{G}_i$ is defined as the kernel of the tautological map $\mathcal{E}_\Delta \twoheadrightarrow \mathcal{L}_i$, $\forall i \in \{1,2\}$, and $\Delta \cong \mathfrak{Z}_d \hookrightarrow \mathfrak{Z}_d \times C$ denotes the graph of the map which remembers the point $x$.

\end{lemma}

\begin{proof} 
As the map $\iota$ is the diagonal embedding, the fact that it is cut out by the section \eqref{eqn:sec diag} is simply the well-known resolution of the diagonal in a projective bundle over an arbitrary base scheme.

\end{proof}

\subsection{Multiplication operators} We are now ready to deduce Proposition \ref{prop:mult intro}, expressing the multiplication by universal Chern classes in terms of the $a$ and $f$ operators. 
\begin{proof}
Let us consider the class
\begin{equation}
\label{eqn:expression}
\frac 1{z-\lambda} \in \fH_{\quot_{d-1,d}} [[z^{-1}]]
\end{equation}
Because of Lemma \ref{lemma:projectivizations}, the push-forward of \eqref{eqn:expression} under $p_+ \times p_C$ is the power series
$$
\left[- \frac {c(V,z+K_C)}{c(\mathcal{E},z+K_C)} \right]_{z^{<0}} = 1 - \frac {c(V,z+K_C)}{c(\mathcal{E},z+K_C)} \in \fH_{\quot_d \times C} [[z^{-1}]]
$$
(the minus sign is due to $\mathcal{L}$ being isomorphic to $\CO(-1)$ on the projectivization \eqref{eqn:quot proj 2}). Pushing forward the resulting class under the bottom map $\Delta_{\text{Quot}} \times \text{Id}_C$ in \eqref{eqn:comm diag} yields the correspondence that realizes 
\begin{equation}
\label{eqn:expression 1}
\text{the operator on } \fH_{\quot_d \times C} \text{ of multiplication by } 1 - \frac {c(V,z+K_C)}{c(\mathcal{E},z+K_C)}.  
\end{equation}
On the other hand, Lemma \ref{lem:diag} implies that the push-forward of the class \eqref{eqn:expression} under the regular embedding $\iota$ is
$$
\frac {c(\mathcal{G}_1,\lambda_2)}{z-\lambda_2} \in \fH_{\mathfrak{Z}_d}[[z^{-1}]]
$$
where $\lambda_2 = c_1(\mathcal{L}_2)$. Because $c(\mathcal{G}_1,z)$ is a polynomial of degree $r-1$ in $z$, then the class above is simply
$$
\left[ \frac {c(\mathcal{G}_1,z)}{z-\lambda_2} \right]_{z^{<0}} \in \fH_{\mathfrak{Z}_d}[[z^{-1}]]
$$
The push-forward of the class above under the map $\pi$ in diagram \eqref{eqn:comm diag} yields the correspondence that realizes
\begin{equation}
\label{eqn:expression 2}
\text{the operator } \left[ a(z)f(z) \Big|_\Delta  \right]_{z^{<0}} \text{ on } \fH_{\quot_d \times C}
\end{equation}
By comparing \eqref{eqn:expression 1} to \eqref{eqn:expression 2}, we see that, on $\fH_{\quot_d \times C}$,
\begin{equation}
\label{eqn:equality series}
\text{the operator of multiplication by }\frac {c(V,z+K_C)}{c(\mathcal{E},z+K_C)} \ \text{ equals } \ \text{Id } - \left[ a(z)f(z) \Big|_\Delta  \right]_{z^{<0}}.
\end{equation}
By taking the coefficient of $z^{-k}$ in the series \eqref{eqn:equality series}, we immediately obtain
\begin{equation}
\label{eqn:equality coefficients}
\Big[ \text{operator of multiplication by } c_k((V-\mathcal{E}) \otimes \mathcal{K}_C^{-1}) \Big] = \sum_{i=0}^{r-1}  a_i f_{r+k-2-i} \Big|_\Delta (-1)^{i-k-1}
\end{equation}
for all $k \geq 1$, which is the statement of Proposition \ref{prop:mult intro}.
\end{proof}

It is elementary to deduce formulas for the operators of multiplication by $c_k(\CE)$ from either \eqref{eqn:equality series} or \eqref{eqn:equality coefficients}, although these formulas are rather cumbersome.

\subsection{Commutators}

Proposition \ref{prop:mult intro} shows the importance of studying the operators $a_i$ and $f_j$ in tandem. We now show Proposition \ref{prop:commute intro}, which for any $i,j \in \{0,\dots,r-1\}$, establishes the commutators of these operators as follows:
\begin{equation}
\label{eqn:comm a and f}
[f_j,a_i] = \delta \cdot \begin{cases} \chi_{i+j,r-1} (-1)^i  + \sum_{s=0}^{i-1} a_s f_{i+j-s-1} (-1)^{i-s} &\text{if } i+j \leq r-1 \\ - \sum_{s=i}^{r-1} a_s f_{i+j-s-1} (-1)^{i-s} &\text{if } i+j \geq r \end{cases}
\end{equation}
where $\chi$ denotes the Kronecker delta function. The LHS is an operator $\fH_{\quot_d} \rightarrow \fH_{\quot_d \times C \times C}$, with $a_i$ and $f_j$ acting in the first and second factor of $C \times C$, respectively. The RHS is a class supported on the diagonal of $C \times C$, due to the presence of $\delta$. 

\begin{proof}

In the subsequent formulas, any product of $k$ series among $e(z), f(z), m(z)$ takes values in $\fH_{C^k}$; with this in mind, $\Delta_{ij}$ denotes the obvious codimension 1 diagonals of $C^k$, while $\delta$ denotes the class of the small (dimension 1) diagonal of $C^k$. We have the following formulas
\begin{align*}
f(w) a(z) &\stackrel{\eqref{eqn:formula a in terms of e}}= f(w) e(z) m(z) \Big|_{\Delta_{23}} \stackrel{\eqref{eqn:rel ef_quot}}= \left[ e(z) f(w) m(z)  \Big|_{\Delta_{13}} - \delta \cdot \frac {h(z) - h(w) }{z-w} m(z) \right]_{z^{\geq 0}} \\
&\stackrel{\eqref{eqn:rel fm_quot}}=  \left[ e(z) m(z) f(w) \Big|_{\Delta_{12}} \cdot \frac {z-w+\delta}{z-w}   - \delta \cdot \frac {h(z) - h(w)}{z-w}m(z) \right]_{z \gg w | w^{< 0}, z^{\geq 0}} = \\
&= a(z) f(w) + \delta \left[ \frac {a(z)f(w)}{z-w} -  \frac {h(z) - h(w) }{z-w} m(z) \right]_{z \gg w | w^{< 0}, z^{\geq 0}} 
\end{align*}
Extracting the coefficient of $z^{r-1-i} (-1)^i w^{-j-1}$ from the expression above yields the following identity for all $i \in \{0,\dots,r-1\}$ and all $j \geq 0$
\begin{equation}
\label{eqn:comm a and f general}
[f_j,a_i] = \delta \sum_{s=0}^{i-1} a_s f_{i+j-s-1} (-1)^{i-s} + \delta \sum_{s=0}^i h_{i+j-r-s+1} m_s (-1)^{i-s}
\end{equation}
where we write
$$
h(z) = \sum_{t=0}^{\infty} \frac {h_t}{z^{r+t}} \qquad \Rightarrow \qquad \frac {h(z)-h(w)}{z-w} = - \sum_{t,t' = 0}^{\infty} \frac {h_{t+t'-r+1}}{z^{t+1}w^{t'+1}}
$$
with the implication that $h_t = 0$ if $t<0$. Since the second sum in the RHS of \eqref{eqn:comm a and f general} is vacuous (respectively equal to $(-1)^i$) when $i+j < r-1$ (respectively $i+j=r-1$), we conclude the case $i+j \leq r-1$ of \eqref{eqn:comm a and f}. On the other hand, if $j < r \leq i+j$, then
$$
\sum_{s=0}^i h_{i+j-r-s+1} m_s (-1)^{i-s} = \sum_{s=0}^{i+j-r+1} h_{i+j-r-s+1} m_s (-1)^{i-s} = (-1)^i \Big[h(z)m(z)\Big]_{z^{-i-j+r-1}} = 
$$
$$
= (-1)^i \left[\text{multiplication by }\frac {c(V,z+K_C)}{c(\CE,z+K_C)}\right]_{z^{-i-j+r-1}} \stackrel{\eqref{eqn:equality coefficients}}= \sum_{s=0}^{r-1} a_s f_{i+j-s-1} (-1)^{i-s-1} 
$$
Combining the formula above with \eqref{eqn:comm a and f general} implies \eqref{eqn:comm a and f} in the case $j < r \leq i+j$.

\end{proof}

\subsection{The proof of Theorem \ref{theorem:fock basis}: a basis for cohomology}

An important observation concerning formula \eqref{eqn:comm a and f} is that all of the indices of $f$ in the right-hand side are smaller than $r$. Thus, formula \eqref{eqn:comm a and f} can be used to successively move any $f_j$ with $j < r$ to the right of any product of $a_i$'s. This is the key observation that will allow the following argument.
 
\begin{proposition}
\label{prop:linearly independent}

The products
$$
\Big\{ a_{i_1} \dots a_{i_d}(\gamma) |0\rangle \Big\}^{r > i_1 \geq \dots \geq i_d \geq 0}_{\gamma \in \text{a fixed }\BQ \text{-basis of } \fH_{C^d,\Sigma}}
$$
are linearly independent elements of $\fH_{\quot_d}$, where $\Sigma \subset S(d)$ is the subgroup of permutations generated by the transpositions $(ij)$ for all $i \neq j$ such that $k_i = k_j$.
    
\end{proposition}

\begin{proof} Suppose we have a non-trivial relation
\begin{equation}
\label{eqn:non-trivial relation}
a_{i_1} \dots a_{i_d}(\gamma) |0\rangle \in \sum_{r > i_1' \geq \dots \geq i_d' \geq 0} \sum_{\gamma' \neq \gamma} \mathbb{Q} \cdot a_{i'_1} \dots a_{i'_d}(\gamma') |0\rangle
\end{equation}
where we assume that the right-hand side only features indices such that either $i_1'+\dots+i_d' < i_1+\dots+i_d$ or $i_1'+\dots+i_d' = i_1+\dots+i_d$ but the partition $(i_1' \geq \dots \geq i_d')$ is lexicographically larger than $(i_1\geq \dots \geq i_d)$. From \eqref{eqn:comm a and f}, we see that as we commute $f_j$ past $a_i$, the answer is equal to $(-1)^i \cdot \chi_{i+j,r-1}$ plus a linear combination of $a_{i'}f_{j'}$ with $i'+j' < i+j$. Thus, if we apply the composed operator (for some $\phi \in \fH_{C^d}$)
\begin{equation}
\label{eqn:the product}
f_{r-1-i_d} \dots f_{r-1-i_1} (\phi) : \fH_{\quot_d} \rightarrow \fH_{\quot_0} = \BQ \cdot |0\rangle
\end{equation}
to both sides of \eqref{eqn:non-trivial relation}, the only terms that survive are the various $(-1)^i \cdot \chi_{i+j,r-1}$'s (this happens because $f_j|0\rangle = 0$, so as we move the various $f_j$'s to the right of the product of $a_i$'s, the end result can be non-zero only if all the $f_j$'s are absorbed by the $a_i$'s through the terms $(-1)^i \cdot \chi_{i+j,r-1}$). However, because we assumed that all partitions that appear with non-zero coefficient in the RHS of \eqref{eqn:non-trivial relation} are strictly smaller than $(i_1\geq \dots \geq i_d)$ lexicograhically, the corresponding classes $a_{i'_1} \dots a_{i'_d}(\gamma') |0\rangle$ are all annihilated by \eqref{eqn:the product}. On the other hand, \eqref{eqn:the product} applied to the LHS of \eqref{eqn:non-trivial relation} is equal to multiplication by $\pm \int \phi \cdot \gamma$, which is non-zero for suitably chosen $\phi$ due to the non-degeneracy of the intersection pairing. We have thus obtained a contradiction to \eqref{eqn:non-trivial relation}.

\end{proof}

\begin{proposition}
\label{prop:span}

The products
$$
\Big\{ a_{i_1} \dots a_{i_d}(\gamma) |0\rangle \Big\}^{r > i_1 \geq \dots \geq i_d \geq 0}_{\gamma \in \text{a fixed }\BQ \text{-basis of } \fH_{C^d,\Sigma}}
$$
linearly span $\fH_{\quot_d}$, where the notation is that of Proposition \ref{prop:linearly independent}.
    
\end{proposition}

\begin{proof} As pointed out in Remark \ref{remark:diagonal}, $\fH_{\quot_d}$ is linearly spanned by universal classes. Any such universal class can be regarded as the result of acting with the operator ``multiplication by a universal class" on the fundamental class
$$
1_{\quot_d} = \frac 1{d!} \cdot \underbrace{a_0 \dots a_0}_{d \text{ times}}(1_{C^d}) |0\rangle \in \fH^0_{\quot_d}
$$
(here we iterate the equality $a_0(1_C)\cdot 1_{\quot_{d-1}} = d \cdot 1_{\quot_d}$, which follows from the fact that the map $\quot_{d-1,d} \rightarrow \quot_d$ is generically a $d$-to-$1$ map). Therefore, we conclude that $\fH_{\quot_d}$ is linearly spanned by classes of the form
$$
m_{k_1} \dots m_{k_t} (\gamma) a_0 \dots a_0 (1_{C^d})|0\rangle
$$
for various $r \geq k_1 \geq \dots \geq k_t > 0$ and $\gamma \in \fH_{C^t}$. However, formula \eqref{eqn:equality coefficients} implies that we can write any class as above as a linear combination of classes of the form
\begin{equation}
\label{eqn:any class}
a_{i_1}f_{j_1} \dots a_{i_t}f_{j_t}(\gamma) a_0 \dots a_0 (1_{C^d})|0\rangle
\end{equation}
for various $i_1,j_1,\dots,i_t,j_t$ and $\gamma \in \fH_{C^{2t}}$. We may now use \eqref{eqn:comm a and f general} to move all the $f$'s to the right of all the $a$'s (recall that any expression where $f_j$ is immediately to the left of $|0\rangle$ vanishes, for all $j \geq 0$). When doing so, one may encounter sums of products of other $m_k$'s in the middle of the product of $a$'s and $f$'s in \eqref{eqn:any class}. Then we reconvert the corresponding $m_k$'s into sums of $a_if_j$'s via \eqref{eqn:equality coefficients}, and repeat the argument; it is easy to see that it terminates in finitely many steps, as the $f$'s and $m$'s can only move closer to $|0\rangle$ throughout the process. Thus, any class \eqref{eqn:any class} can be successively written as a linear combination of the classes $a_{i_1} \dots a_{i_d}(\gamma) |0\rangle$, just as we needed to prove.

\end{proof}
Together, Propositions \ref{prop:linearly independent} and \ref{prop:span} imply Theorem \ref{theorem:fock basis}. Note that our proof of Proposition \ref{prop:span} also holds at the level of Chow groups, but our proof of Proposition \ref{prop:linearly independent} does not.

\section{Appendix: the quadruple moduli space}

\subsection{The moduli space of quadruples}
\label{sub:quadruples}

In the preceding sections, we proved commutation relations such as \eqref{eqn:commute a intro} and \eqref{eqn:rel ee_quot} using the forgetful map $\quot_{d,d+1,d+2} \rightarrow \quot_{d,d+2}$ (which leaves out the middle sheaf in \eqref{eqn:def two-step nested quot scheme}). In the current Section, we will show an alternative method to obtain the aforementioned formulas, using a blow-up instead. 

\begin{definition}
\label{def:quadruples}

Consider the moduli space $\fY_{d,d+1,d+2}$ parameterizing quadruples
\begin{equation}
\label{eqn:quadruple}
 \xymatrix{& E' \ar@{^{(}->}[rd]^{y} & \\
E'' \ar@{^{(}->}[ru]^{x} \ar@{^{(}->}[rd]_{y}
& & E  \subset  V \\
& \widetilde{E}' \ar@{^{(}->}[ru]_{x} &} 
	\end{equation}
with the right-most inclusion having colength $d$.

\end{definition}

We have projection maps
$$
\pi^\uparrow, \pi^\downarrow : \fY_{d,d+1,d+2} \rightarrow \quot_{d,d+1,d+2}
$$
that forget $\widetilde{E}'$ and $E'$, respectively. It is clear that the maps above are projective, and therefore so is $\fY_{d,d+1,d+2}$ (in fact, the natural analogue of \cite[Proposition 2.37]{negut4} shows that the map $\pi^\uparrow$ is l.c.i.; as $\quot_{d,d+1,d+2}$ is smooth, this implies that $\fY_{d,d+1,d+2}$ is l.c.i.). In the following subsections, we will prove the following stronger result.

\begin{proposition}
\label{prop:blow up}

The map $\pi^\uparrow$ is the blow-up of the subvariety 
\begin{equation}
\label{eqn:divisor}
Z = \Big\{x = y \text{ and } E/E'' \cong \mathbb{C}_x \oplus \mathbb{C}_x \Big\} 
\end{equation}
of $\quot_{d,d+1,d+2}$ (with closed points as in \eqref{eqn:def two-step nested quot scheme}). Moreover, $Z$ is a smooth codimension 2 subvariety of $\quot_{d,d+1,d+2}$, hence $\fY_{d,d+1,d+2}$ is smooth of dimension $r(d+2)$.

\end{proposition}

\subsection{} We begin by investigating the subvariety $Z$ of \eqref{eqn:divisor}, and proving its smoothness.

\begin{lemma}
\label{lem:divisor}

The subvariety $\tZ = \{x = y\}$ is a smooth divisor in $\quot_{d,d+1,d+2}$. One dimension down, the subvariety $Z$ of \eqref{eqn:divisor} is a smooth divisor in $\tZ$.

\end{lemma}

\begin{proof}

Since the subvariety of $\quot_{d,d+1,d+2}$ determined by $\{x=y\}$ is none other than 
$$
\widetilde{Z} = \mathbb{P}_{\quot_{d,d+1}}(\mathcal{E}'_{\Delta})
$$
(where we recall that $\Delta \cong \quot_{d,d+1} \hookrightarrow \quot_{d,d+1} \times C$ is the graph of $p_C$) it is clear that it is a smooth divisor in $\quot_{d,d+1,d+2}$. Moreover, the subvariety
$$
Z = \mathbb{P}_{\quot_{d,d+1}}(\mathcal{G}) \hookrightarrow \widetilde{Z}
$$
(induced by the short exact sequence $0 \rightarrow \mathcal{L} \otimes \mathcal{K}_C \rightarrow \mathcal{E}'_\Delta \rightarrow \mathcal{G} \rightarrow 0$ of \eqref{eqn:tor nested}) is clearly smooth and one dimension lower. By definition, the projectivization $Z$ defined above parameterizes those triples $E'' \stackrel{x}\hookrightarrow E'  \stackrel{x}\hookrightarrow E$ such that the composition
\begin{equation}
\label{eqn:map above}
E_x/E'_x \otimes \underbrace{\mathfrak{m}_x/\mathfrak{m}_x^2}_{K_{C,x}} \xrightarrow{f} E'_x \rightarrow E'_x/E''_x
\end{equation}
vanishes. However, the map $f$ above is defined by taking an arbitrary (local around $x$) section $v$ of $E$, an arbitrary coordinate $t$ on $C$ (also local around $x$) and noting that $tv$ is a section of $E'$. The condition that the composition \eqref{eqn:map above} vanishes implies that $tv$ is actually a section of $E''$, which means that the length 2 quotient sheaf $E/E''$ is annihilated by $t$. This is precisely equivalent to the definition of $Z$ in \eqref{eqn:divisor}. 

\end{proof}

\subsection{} 

Let us now prove the parts of Proposition \ref{prop:blow up} that pertain to $\fY_{d,d+1,d+2}$, specifically the fact that it is the blow-up of the locus $Z \hookrightarrow \quot_{d,d+1,d+2}$. To this end, consider the subvarieties
\begin{align}
&\text{Exc} = \Big\{ x = y \text{ and } E/E'' \cong \BC_x \oplus \BC_x \Big\} \label{eqn:exceptional} \\
&\text{Diag} = \Big\{x = y \text{ and } E' = \tE' \Big\}  \label{eqn:diagonal}
\end{align}
of $\fY_{d,d+1,d+2}$. We have
\begin{equation}
\label{eqn:two divisors}
\Big\{x = y\Big\} = \text{Exc} \cup \text{Diag} \subset \fY_{d,d+1,d+2}
\end{equation}
because if $x = y$ but $E/E'' \not \cong \BC_x \oplus \BC_x$, then $E' = \tE'$ in any quadruple \eqref{eqn:quadruple}. 

\begin{lemma}
\label{lem:weil}

The subvarieties \emph{Exc} and \emph{Diag} are smooth Weil divisors of $\fY_{d,d+1,d+2}$, and their intersection is codimension $2$ and naturally isomorphic to the locus $Z$ of Lemma \ref{lem:divisor}.

\end{lemma}

\begin{proof} It is clear that the map $\pi^\uparrow$ is an isomorphism outside $\text{Exc}$, and that
$$
\pi^\uparrow : \text{Exc} \rightarrow Z
$$
is a $\BP^1$-fibration. Together with Lemma \ref{lem:divisor}, this implies that $\text{Exc}$ is smooth of dimension $r(d+2)-1$. On the other hand, $\text{Diag}$ is isomorphic to the subvariety $\tZ$ of Lemma \ref{lem:divisor}, and so it is also smooth of dimension $r(d+2)-1$. It is clear that the intersection of $\text{Exc}$ and $\text{Diag}$ is isomorphic to $Z$, which is smooth of dimension $r(d+2)-2$. Putting all the above remarks together, we conclude that $\fY_{d,d+1,d+2}$ has dimension $r(d+2)$, hence $\text{Exc}$ and $\text{Diag}$ are Weil divisors satisfying the conditions stated in the Lemma.
    
\end{proof}

\subsection{Line bundles on $\fY_{d,d+1,d+2}$} 
\label{sub:line}

We will now prove that the Weil divisors \text{Exc} and \text{Diag} are actually Cartier. To this end, let us consider the line bundles
$$
\mathcal{L}_1, \mathcal{L}_2, \widetilde{\mathcal{L}}_1,  \widetilde{\mathcal{L}}_2 \quad \text{on }\fY_{d,d+1,d+2}
$$
whose fibers in the notation of \eqref{eqn:quadruple} are $E'_x/E''_x$, $E_y/E'_y$, $\widetilde{E}'_y/E''_y$, $E_x/\widetilde{E}'_x$, respectively. 

\begin{lemma}
\label{lemma:divisors on y}

On $\fY_{d,d+1,d+2}$, there exist maps of line bundles 
$$
\CL_1 \rightarrow \tCL_2 \qquad \text{and} \qquad \tCL_1 \rightarrow \CL_2
$$
whose associated Weil divisor is $\emph{Diag}$, and maps of line bundles
$$
\tCL_2 \rightarrow \CL_1(\Delta) \qquad \text{and} \qquad \CL_2 \rightarrow \tCL_1(\Delta)
$$
whose associated Weil divisor is $\emph{Exc}$ (above, we write $\CO(\Delta)$ for the pull-back to $\fY_{d,d+1,d+2}$ of the line bundle corresponding to the diagonal in $C \times C$).

\end{lemma}

\begin{proof} The natural inclusions in the square \eqref{eqn:quadruple} yield maps of line bundles
\begin{align*}
&\CL_1 \rightarrow \tCL_2 \quad \text{induced by } E'_x/E''_x \rightarrow E_x/\tE'_x \\
&\tCL_1 \rightarrow \CL_2 \quad \text{induced by } \tE'_y/E''_y \rightarrow E_y/E'_y
\end{align*}
These maps of line bundles vanish precisely when $x=y$ and $E' = \tE'$, i.e. on the locus \text{Diag}. Meanwhile, choosing local sections $f \in \Gamma(C \times C, \CO(\Delta))$ allows us to construct maps of line bundles
\begin{align*}
&\tCL_2 \rightarrow \CL_1(\Delta) \quad \text{induced by } E_x/\tE'_x \xrightarrow{\text{multiplication by }f(x,-)} E'_x/E''_x\\
&\CL_2 \rightarrow \tCL_1(\Delta) \quad \text{induced by } E_y/E'_y \xrightarrow{\text{multiplication by }f(-,y)} \tE'_y/E''_y
\end{align*}
These maps of line bundles vanish precisely when $x=y$ and $E/E'' \cong \BC_x \oplus \BC_x$ (the latter condition being precisely equivalent to $f$ annihilating the length 2 quotient sheaf $E/E''$), i.e. on the locus \text{Exc}. 

\end{proof}

\begin{proof} \emph{of Proposition \ref{prop:blow up}:} Recall from the proof of Lemma \ref{lem:weil} that $\pi^\uparrow$ is an isomorphism outside of $Z$, while
$$
\pi^{\uparrow,-1}(Z) = \text{Exc}
$$
Since we have seen in Lemma \ref{lemma:divisors on y} that \text{Exc} is a Cartier divisor, there exists a morphism
$$
\fY_{d,d+1,d+2} \xrightarrow{\nu} \text{Bl}_Z\quot_{d,d+1,d+2}
$$
The blow-up on the right is smooth (seeing as how both $\quot_{d,d+1,d+2}$ and $Z$ are smooth, see Lemma \ref{lem:divisor}). The morphism $\nu$ is a bijection on closed points, because the fibers of both domain and codomain of $\nu$ over any point of $Z$ is a copy of $\BP(\Gamma(C,E/E''))$. We thus conclude that $\nu$ is an isomorphism, as we needed to show.
    
\end{proof}

As an immediate consequence of Proposition \ref{prop:blow up}, we have the formulas
\begin{equation}
\label{eqn:birational}
\pi^\uparrow_*(1) = \pi^\downarrow_*(1) = 1 
\end{equation}
as well as
\begin{equation}
\label{eqn:exc div}
\pi^\uparrow_*(\varepsilon) = \pi^\downarrow_*(\varepsilon) = 0 
\end{equation}
where $\varepsilon = [\text{Exc}] \in \fH^2_{\fY_{d,d+1,d+2}}$ is the class of the exceptional divisor. The formulas above allow us to replace various cohomology classes on $\quot_{d,d+1,d+2}$ by their pull-backs to $\fY_{d,d+1,d+2}$; in light of \eqref{eqn:birational}, this replacement yields the same correspondences on quot schemes. We will use this approach in Subsection \ref{sub:e-e} to produce an alternative proof for relation \eqref{eqn:rel ee_quot}.

\subsection{The cohomology of $\fY_{d,d+1,d+2}$} Recall the line bundles $\{\CL_i,\tCL_i\}_{i \in \{1,2\}}$ of Subsection \ref{sub:line}, and let
$$
\lambda_i = c_1(\CL_i), \ \tlambda_i = c_1(\tCL_i) \in \fH^2_{\fY_{d,d+1,d+2}}
$$
for all $i \in \{1,2\}$. As
$$
[\mathcal{L}_1] + [\mathcal{L}_2] = [ \widetilde{\mathcal{L}}_1] + [ \widetilde{\mathcal{L}}_2] 
$$
holds in the Grothendieck group of $\fY_{d,d+1,d+2}$, we have the identities
\begin{equation}
\label{eqn:chern y}
\lambda_1+\lambda_2 = \tlambda_1 + \tlambda_2 \quad \text{and} \quad \lambda_1\lambda_2 = \tlambda_1 \tlambda_2
\end{equation}
Moreover, Lemma \ref{lemma:divisors on y} immediately implies the identity 
\begin{equation}
\label{eqn:divisors on y}
\tlambda_2 - \lambda_1 = \lambda_2 - \tlambda_1 = \delta - \varepsilon
\end{equation}
in $\fH_{\fY_{d,d+1,d+2}}$, where $\delta$ is the class of the diagonal $\{x=y\}$ and $\varepsilon$ is the class of \text{Exc} (note that $\delta - \varepsilon = [\text{Diag}]$, as an immediate consequence of \eqref{eqn:two divisors}).

\begin{proposition}
\label{prop:h4}

We have the following identity in $\fH^4_{\fY_{d,d+1,d+2}}$
\begin{equation}
\label{eqn:divisors on y squared}
(\lambda_i - \tlambda_1)(\lambda_i - \tlambda_2) = (\lambda_1 - \tlambda_i)(\lambda_2-\tlambda_i) = 0
\end{equation}
for all $i \in \{1,2\}$.

\end{proposition}

\begin{proof}  Because the line bundles $\CL_1$ and $\tCL_1$ (respectively $\CL_2$ and $\tCL_2$) are isomorphic on $\text{Diag}$, we obtain the identity
\begin{equation}
\label{eqn:identity with delta}
[\text{Diag}](\lambda_1 - \tlambda_1) = [\text{Diag}](\lambda_2 - \tlambda_2) = 0
\end{equation}
Substituting $\tlambda_2 - \lambda_1$ or $\lambda_2 - \tlambda_1$ for $[\text{Diag}] = \delta - \varepsilon$ in the above formula, which we are allowed to do because of \eqref{eqn:divisors on y}, yields formula \eqref{eqn:divisors on y squared}.

\end{proof}

\subsection{An alternative proof of the commutator of $e(z)$ and $e(w)$}
\label{sub:e-e}

We will now use the language of quadruple moduli spaces to produce an alternate proof of formula \eqref{eqn:rel ee_quot coefficient}, specifically the following equality for all $k,l \geq 0$ (see \eqref{eqn:rel ee_quot coefficient}):
\begin{equation}
\label{eqn:ee explicit}
e_{k+1} e_l - e_k e_{l+1} + \delta e_ke_l = e_l e_{k+1} - e_{l+1} e_k - \delta e_l e_k
\end{equation}
As a correspondence, the LHS of \eqref{eqn:ee explicit} is given by the cohomology class
$$
\lambda_1^k\lambda_2^l (\lambda_1-\lambda_2+\delta) \in \fH_{\quot_{d,d+1,d+2}}
$$
where $\lambda_1 = c_1(\CL_1)$ and $\lambda_2 = c_1(\CL_2)$. By \eqref{eqn:birational} and \eqref{eqn:exc div}, we can equivalently say that the LHS of \eqref{eqn:ee explicit} is given by the correspondence
\begin{equation}
\label{eqn:the left}
\lambda_1^k\lambda_2^l (\lambda_1-\lambda_2+\delta-\varepsilon) \in \fH_{\fY_{d,d+1,d+2}}
\end{equation}
Similarly, the RHS of \eqref{eqn:ee explicit} is given by the cohomology class
\begin{equation}
\label{eqn:the right}
\tlambda_1^l\tlambda_2^k (\tlambda_2-\tlambda_1-\delta+\varepsilon) \in \fH_{\fY_{d,d+1,d+2}}
\end{equation}
It therefore suffices to prove that the cohomology classes \eqref{eqn:the left} and \eqref{eqn:the right} are equal on $\fY_{d,d+1,d+2}$. However, this follows from the chain of equalities
$$
\text{the class \eqref{eqn:the left}} \stackrel{\eqref{eqn:divisors on y}}= \lambda_1^k \lambda_2^l(\tlambda_2-\lambda_2) \stackrel{\eqref{eqn:divisors on y squared}}= \tlambda_2^k \tlambda_1^l(\tlambda_2-\lambda_2) \stackrel{\eqref{eqn:divisors on y}}= \text{the class \eqref{eqn:the right}}
$$

\begin{remark}

The moduli space of quadruples can also be used to provide geometric proofs of relations \eqref{eqn:commute a intro} and \eqref{eqn:rel ef_quot}. Indeed, in the situation where quot schemes on curves are replaced by moduli spaces of sheaves on surfaces, the analogous spaces of quadruples have been used to obtain interesting relations at the level of $K$-theory groups (\cite{negut4}) and derived categories (\cite{Yu}).

\end{remark}

\section{Appendix: comparing cohomology with Chow groups}
\label{sec:chow}

\subsection{Chow groups}
\label{sub:chow}

In the present section, we provide details on the generalization of our main results to Chow groups with rational coefficients
$$
\fA_X = \Big(\BQ\text{-algebraic cycles on }X \text{ modulo rational equivalence}\Big),
$$
considered throughout this section for smooth projective varieties $X$. Chow groups admit the same functoriality properties as ordinary cohomology groups, and the language of correspondences holds in Chow just as it did for cohomology in Subsection \ref{sub:corr}. All the representation theoretic constructions in the present paper involve Chern classes of locally free sheaves and fundamental classes of algebraic subvarieties, which exist in Chow groups just as they do in cohomology groups. Thus, the following results hold at the level of Chow groups, and the proofs are virtually identical with the one we already provided at the level of cohomology groups (we abbreviate $\fA_{\quot} = \oplus_{d = 0}^{\infty} \fA_{\quot_d(V)}$ etc).
$$
\emph{Proposition \ref{prop:commute a intro}:} \ \quad \exists \ a_k : \fA_{\quot} \rightarrow \fA_{\quot \times C} \text{ defined as in \eqref{eqn:formula a intro}, such that } a_k  a_{k'} = a_{k'} a_k
$$
$$
\emph{Theorem \ref{theorem:yangian}, Corollary \ref{cor:yangian}:} \ \exists \ e_k, f_k, m_i : \fA_{\quot} \rightarrow \fA_{\quot\times C} \text{ defined as in \eqref{eqn:formula e intro}, \eqref{eqn:formula f intro}, \eqref{eqn:operator m intro},}
$$
$$
\text{which induce an action } Y_{\hbar}^r(\mathfrak{sl}_2) \curvearrowright \fA_{\quot}
$$
(the notion of action is the same as in Definition \ref{def:action}, but replacing $\fH$ by $\fA$ everywhere)
$$
\emph{Proposition \ref{prop:mult intro}:} \ \text{ the product operator } \fA_{\quot} \xrightarrow{\cdot  c_k((V-\CE)\otimes \CK_C^{-1})} \fA_{\quot \times C} \text{ is given by} 
$$
$$
\sum_{i=0}^{r-1} a_i f_{r+k-i-2} \Big|_\Delta (-1)^{i-k}
$$
$$
\emph{Proposition \ref{prop:commute intro}:} \ \text{ we have the following equality of operators } \fA_{\quot} \rightarrow \fA_{\quot \times C \times C}
$$
$$
[f_j,a_i] = \delta \cdot\begin{cases} \chi_{i+j,r-1} (-1)^i + \sum_{s=0}^{i-1} a_s f_{i+j-s-1}  (-1)^{i-s} &\text{if } i+j \leq r-1 \\ - \sum_{s=i}^{r-1} a_s f_{i+j-s-1} (-1)^{i-s} &\text{if } i+j \geq r \end{cases} 
$$
where $\delta$ denotes the class of the diagonal in the Chow group of $C \times C$. 

We believe that Theorem \ref{theorem:fock basis} holds on the level of Chow groups as well, but our arguments do not establish it because they use the non-degeneracy of the intersection pairing (in Proposition \ref{prop:linearly independent}). Our proof still shows nevertheless that the elements
$$
\Big\{ a_{k_1}\dots a_{k_d}(\gamma)|0\rangle \Big\}^{r > k_1 \geq \dots \geq k_d \geq 0}_{\gamma \in \BQ\text{-basis of }\fA_{C^d,\Sigma}}
$$
linearly span $\fA_{\quot_d}$.

\subsection{Comparison}
\label{sub:comparison}

Despite the many parallels between cohomology and Chow groups, one property of cohomology that does not extend to Chow groups is the K\"unneth isomorphism
$$
\fA_{X \times Y} \not \cong \fA_X \otimes \fA_Y
$$
for general smooth projective varieties $X$ and $Y$. The same failure of the K\"unneth isomorphism holds in algebraic $K$-theory (as studied in our follow-up paper \cite{mn}) and even in cohomology with integer coefficients. This is the reason why we use the notion of action in Definition \ref{def:action}: the natural target for all our correspondences is $\quot \times C$, which in general cannot be ``split up" into a contribution from $\quot$ and a contribution from $C$.

However, in usual cohomology with $\BQ$-coefficients, we do have K\"unneth isomorphisms
$$
\fH_{X \times Y} \cong \fH_X \otimes \fH_Y
$$
for all smooth projective varieties $X$ and $Y$. With this in mind, let us fix a $\BQ$-basis $1 = \gamma_1, \gamma_2, \dots,\gamma_t$ of $\fH_C$ and let us consider an action as in Definition \ref{def:action}. The action map \eqref{eqn:action map}
\begin{equation}
\label{eqn:action map kunneth}
Y_{\hbar}^r(\mathfrak{sl}_2) \xrightarrow{\Phi} \text{Hom} (\fH_{\quot}, \fH_{\quot} \otimes \fH_C)
\end{equation}
can be written as 
\begin{equation}
\label{eqn:phi coefficients}
\Phi(x) = \sum_{i=1}^t \Phi(x)^i \otimes \gamma_i
\end{equation}
for various $\BQ$-linear maps
\begin{equation}
\label{eqn:coefficients}
\Phi(x)^i \colon \fH_{\quot} \rightarrow \fH_{\quot}
\end{equation}
defined for all $x \in Y_{\hbar}^r(\mathfrak{sl}_2)$. Conditions (1)-(4) of Definition \ref{def:action} are easily seen to translate into the following properties of the endomorphisms \eqref{eqn:coefficients}
\begin{align}
&(1): \Phi(1)^i = \delta_{i1} \label{eqn:1}, \qquad \qquad \qquad \qquad \qquad \qquad \ \ \forall i \in \{1,\dots,t\} \\
&(2): \Phi(\hbar x)^i = \sum_{j=1}^t a^i_j \Phi(x)^j, \qquad \qquad \qquad \qquad \ \forall i \in \{1,\dots,t\} \label{eqn:2} \\
&(3): \Phi(xy)^i = \sum_{j,k=1}^t b^i_{jk} \Phi(x)^j\phi(y)^k \label{eqn:3}, \qquad \qquad \quad \ \forall i \in \{1,\dots,t\} \\
&(4): \sum_{i=1}^t c^{jk}_i \Phi\left(\frac {[x,y]}{-\hbar} \right)^i = [\Phi(x)^j,\Phi(y)^k], \qquad \forall j,k \in \{1,\dots,t\} \label{eqn:4} 
\end{align}
where the structure constants $a^i_j, b^i_{jk}, c^{jk}_i \in \BQ$ are given by
$$
\hbar \gamma_j = \sum_{j=1}^k a^i_j \gamma_i \in \fH_C
$$
$$
\gamma_j \gamma_k = \sum_{j,k=1}^t b^i_{jk} \gamma_i \in \fH_C
$$
$$
\Delta_*(\gamma_i) = \sum_{j,k=1}^t c^{jk}_i \gamma_j \boxtimes \gamma_k \in \fH_{C \times C}
$$
Relations \eqref{eqn:1}-\eqref{eqn:4} precisely mimic those of \cite[Theorem 3.1]{nakajima}, where Nakajima defined a celebrated Heiseberg algebra action on the cohomology of Hilbert schemes of points on surfaces (an equivalent formulation was given in \cite{grojnowski}). Thus, the notion in Definition \ref{def:action} has the two-fold advantage of packaging actions on cohomology of moduli spaces in an index-free way, but also encoding strictly more information than endomorphisms such as \eqref{eqn:coefficients} when cohomology is replaced by theories without K\"unneth decompositions (such as Chow groups or algebraic $K$-theory).

\end{document}